\documentclass[10pt]{article}
\usepackage{graphicx}
\usepackage{amsfonts}
\usepackage{amsmath}
\usepackage{amssymb}
\usepackage{url}
\usepackage{fancyhdr}
\usepackage{indentfirst}
\usepackage{enumerate}
\usepackage{authblk}

\usepackage{dsfont}
\usepackage[colorlinks=true,citecolor=blue]{hyperref}
\usepackage{amsthm}
\usepackage{color}
\usepackage{natbib}
\usepackage{comment}
\usepackage{capt-of}
\usepackage{multirow}
\usepackage{subfig}
\def\d{\mathrm{d}}

\newcommand{\D}{\mathcal {D}}

\newcommand{\X}{\mathcal {X}}

\newcommand{\E}{\mathbb{E}}

\newcommand{\F}{\mathcal{F}}

\newcommand{\R}{\mathbb{R}}

\newcommand{\p}{\mathbb{P}}

\newcommand{\Y}{\mathcal{Y}}
\newcommand{\M}{\mathcal{M}}

\renewcommand{\(}{\left(}
\renewcommand{\)}{\right)}
\renewcommand{\[}{\left[}
\renewcommand{\]}{\right]}

\renewcommand{\geq}{\geqslant}
\renewcommand{\leq}{\leqslant}
\renewcommand{\epsilon}{\varepsilon}

\renewcommand{\cdots}{\dots}

\theoremstyle{plain}
\newtheorem{theorem}{Theorem}

\newtheorem{lemma}{Lemma}
\newtheorem{proposition}{Proposition}
\theoremstyle{definition}
\newtheorem{definition}{Definition}
\newtheorem{example}{Example}

\newtheorem{assumption}{Assumption}
\theoremstyle{remark}
\newtheorem{remark}{Remark}

\newcommand{\cet}{\begin{center}}
\newcommand{\ecet}{\end{center}}

\allowdisplaybreaks
\usepackage{setspace}

\begin{document}
	
\title{PSAHARA Utility Family: Modeling Non-monotone Risk Aversion and Convex Compensation in Incomplete Markets}

\author{
Yang Liu
\thanks{
Email: \texttt{yangliu16@cuhk.edu.cn}
}
\quad\quad\quad\quad\quad\quad\quad\quad 
Zhenyu Shen
\thanks{Corresponding Author. 
Email: \texttt{zhenyushen@link.cuhk.edu.cn}
}

{\small 
School of Science and Engineering, The Chinese University of Hong Kong, Shenzhen
, China. 
}

}
\date{}

\maketitle

\begin{abstract}

In hedge funds, convex compensation schemes are adopted to stimulate a high-profit performance for portfolio managers. In economics, non-monotone risk aversion is proposed to argue that individuals may not be risk-averse when the wealth level is low. Combining these two ingredients, we study the optimal control strategy of the manager in incomplete markets. Generally, we propose a wide family of utility functions, the piecewise symmetric asymptotic hyperbolic absolute risk aversion (PSAHARA) utility, to model the two ingredients, containing both non-concavity and non-differentiability as some abnormalities. Technically, we propose an additional assumption and prove concavification techniques of non--concave utility functions with a left unbounded domain in incomplete markets. Next, we derive an explicit optimal control for the family of PSAHARA utilities. This control is expressed into a unified four-term structure, featuring the asymptotic Merton term.
Furthermore, we provide a detailed asymptotic analysis and numerical illustration of the optimal portfolio. 
We obtain several key insights, including that 
the convex compensation still induces a great risk-taking behavior in the case that the preference is modeled by SAHARA utility. Finally, we conduct a real-data analysis of the U.S. stock market under the above model and conclude that the PSAHARA portfolio is very risk-seeking and leads to a high return and a high volatility.

\end{abstract}

\section{Introduction}
Compensation incentive schemes are commonly adopted to share profits between the portfolio manager and the investors in hedge funds. A typical setting in the financial industry is the ``2--20'' scheme---$2\%$ management fee of the fund value and $20\%$ performance fee of the excess profit, where the latter usually takes the form of a call option; see \cite{C2000}. Such a compensation scheme gives the manager a strong incentive to chase a good performance. Mathematically, this scheme is a piecewise linear convex function and hence is referred to as a convex compensation scheme (formally in Eq. \eqref{eq:incentive contract}). Multiple studies have explored the impact of incentive options on the investment strategies of fund managers; see \cite{BKP2004}, \cite{HJ2007}, 
\cite{BS2014}, \cite{HK2018} and \cite{LL2020}.
In these studies, fund managers are typically assumed to have constant relative risk aversion (CRRA) or hyperbolic absolute risk aversion (HARA) types of utilities or S-shaped utilities with piecewise power function variants. 

However, the CRRA or HARA utilities induce a monotone absolute risk-aversion function (ARA$(\cdot)$, formally given in Definition \ref{def:SAHARA}), which means that the individual will be very risk-averse below a low wealth level. But intuitively, people are relatively indifferent about a slight loss below some wealth level, implying that the individual may be less risk-averse below this threshold. As a result, the absolute risk-aversion function may not be monotone, and 
a new utility function, the so-called symmetric asymptotic hyperbolic absolute risk aversion (SAHARA) utility, is proposed in \cite{CPV2011}. 
This utility is explicit and well exhibits the above feature.
Further, the utility is defined on the whole real line (the wealth could be arbitrarily negative). Together with a CARA utility (Constant Absolute Risk Aversion; i.e., exponential functions), the SAHARA utility acts as an alternative to the CRRA and HARA utilities (i.e., power functions) to help characterize the preference on the arbitrary wealth level. The SAHARA utility is further adopted and studied in, e.g., \cite{CNS2021}, \cite{SZ2021} and \cite{CLH2023}.

If the portfolio manager has a SAHARA preference, the actual utility is the composition of the SAHARA preference and the convex compensation scheme. This actual utility may not be of SAHARA on the whole domain and become complicated. To model this actual utility and further enhance the generality, we introduce a wide family of piecewise SAHARA (PSAHARA) utility functions. Roughly speaking, a utility function in the PSAHARA family satisfies that there is a partition of the real line making the utility SAHARA on each interval (while the different sections allow different parameter settings). Starting from this point, our contribution is fourfold.

First, we design a unified parameterization and introduce the PSAHARA utility family (Definition \ref{def:PSAHARA utility}). Current studies on continuous-time portfolio selection adopt various criterion to evaluate risk and return, e.g., the risk measure in \cite{HJZ2015}, the S-shape utility in \cite{DZ2020}, and the performance ratio in \cite{GLX2023}. Here we capture the individual preference by the PSAHARA utility family, which distinguishes from the commonly used CRRA, CARA, and HARA utilities by 
the non-monotone risk aversion characteristics inherited from the SAHARA utility, allowing for studies on more flexible wealth levels and more complicated risk-taking behaviors. In the PSAHARA family, we emphasize the role of various threshold wealth levels and argue that the individual becomes less risk-averse below each threshold. By introducing some properties on the transformation-invariance of the PSAHARA utility, this wide family incorporates the above objective (composition of the SAHARA preference and the piecewise linear convex compensation) in hedge fund management as a motivating example (Propositions \ref{prop:generality}--\ref{prop:enve}). The utility family may promote other application contexts of the utility theory. 

Second, we formulate the continuous-time portfolio selection problem in incomplete Black--Scholes markets and obtain an explicit optimal control for the general PSAHARA utility (Theorems \ref{thm:optimal wealth}--\ref{optimal portfolio}). As suggested above, the PSAHARA utility may not be of SAHARA on the whole domain, and may not necessarily be concave or differentiable. This creates difficulties in solving an explicit formula of the portfolio. In the proofs, we explain these issues in details.
Technically, in Theorem \ref{thm:optimal wealth} and Proposition \ref{prop:developed optimal condition}, we apply and extend the martingale and duality method (\cite{KLSX1991}) and rigorously discuss the concavification techniques (\cite{C2000}) in the incomplete markets to obtain the optimal portfolio. To rigorously apply the martingale and duality method, we propose an additional assumption (Assumption \ref{assp:assumption 1}), which guarantees the square integrability of the portfolio control processes; see Section \ref{sec:technical discussions} for details.
The optimal control is expressed in a ``partial" feedback form, including parts of the optimal wealth and a unique Lagrange multiplier ($y^*$ in later contexts). We name the four components of the optimal portfolio by their mathematical and economic implications: the asymptotic Merton term, the risk adjustment term (due to the scale parameter in PSAHARA utility), the first-order risk aversion term (due to non-differentiability), and the loss aversion term (due to threshold wealth levels). 

Third, we conduct a comprehensive asymptotic analysis of the optimal control under the general PSAHARA utility family to study the limiting behaviors of the optimal wealth and portfolio (Theorem \ref{thm:asymptotic analysis}). The asymptotic approach is inspired by \cite{LL2024}, but the implementation of the PSAHARA utility family is completely different and technical. 
In addition, we numerically visualize the optimal control dynamics under some examples of PSAHARA utilities. The asymptotic and numerical studies demonstrate that the optimal risky investment percentage tends to the well-known Merton ratio as the wealth grows to infinity; we summarize it as the asymptotic Merton term. The risky investment percentage tends to the negative Merton ratio as the wealth decreases to negative infinity, indicating the fact that people tend to be less risk-averse when their wealth levels fall negatively low. Further, the risky investment percentage tends to infinity as the wealth level tends to zero, suggesting that the manager is risk-seeking at a low wealth level (0).

Fourth, we conduct an empirical study on the motivating example in hedge fund management. We first give the corresponding optimal portfolio (Theorem \ref{thm:incentive portfolio}). We apply this portfolio strategy with the real data of the U.S. stock market, using different estimation methods of the volatility process $\{\pmb{\sigma}_t \}_{0\leq t \leq T}$. 
We analyze the Sharpe ratio of investment 
and find that there exists a ``gambling'' behavior on the linear segments of the composed utility. This means that even if the manager has a SAHARA utility, the convex compensation indeed induces a great risk-taking behavior (which coincides with the classic result of \cite{C2000}). We further find a two-peak pattern of the corresponding Sharpe ratio, meaning that the PSAHARA portfolio leads to a high return and a high volatility.


We compare our novelties with the literature. In the study of incomplete markets, \cite{P1987} pioneers the study on optimal investment and adopts the martingale analysis approach. This approach is developed by \cite{HP1991} and \cite{KLSX1991}. These works mainly focus on the existence of optimal investment strategies, while we technically 
get rid of the assumption of the left bounded domain of non--concave utility functions in Proposition \ref{prop:developed optimal condition} and provide an explicit optimal portfolio in incomplete markets in Theorem \ref{optimal portfolio}. 
In addition, 
\cite{LLMV2024} proposes a family of utility functions, ``piecewise HARA (PHARA) utility", which means that the utility is of HARA on each part of the domain. Our PSAHARA utility family shares a similar logic of generalization, but the SAHARA utility has a more complicated expression and the portfolio is more subtle to obtain. 
Furthermore, as the SAHARA utility can (asymptotically) reduce to a HARA utility, our PSAHARA utility contains the big family of the PHARA utility. Technically, the corresponding optimal control formula in Theorem \ref{optimal portfolio} can exactly reduce to the optimal control of the PHARA portfolio in Theorem 1 of \cite{LLMV2024}. This makes the portfolio formula consistent and unified to implement in future research and financial practice. 

This paper is structured as follows. In Section \ref{sec:motivation}, we introduce our motivating example. We give the definitions of the SAHARA utility and incentive contracts. The model settings are introduced in Section \ref{sec:model setting}. In Section \ref{sec:optimal portfolio}, we present the explicit formula of the optimal portfolio, which is the main theorem of our study. Section \ref{sec:technical discussions} shows our main technical contribution. We conduct the asymptotic and numerical analysis in Section \ref{sec:analysis}. Finally, we revisit our motivating example and conduct empirical studies on the corresponding optimal portfolio in Section \ref{sec:application}. All the proofs are included in Appendix \ref{proofs}. Methods of the empirical study are in Appendix \ref{appendix:empirical}.

\section{Motivating Example in Hedge Funds}
\label{sec:motivation}
In this section, we propose a motivating example and give an analytical model of non-monotone risk aversion and convex compensation.  
The non-monotone risk aversion is depicted by the SAHARA utility family (\cite{CPV2011}). It characterizes complex risk attitudes and show distinct features compared to the widely studied CRRA, CARA and HARA utilities.
\begin{definition}[SAHARA utility function]\label{def:SAHARA}
A utility function $U$ with the domain $\R$ is of the SAHARA family if its absolute risk aversion function $\text{ARA}\left( x \right) = - U^{''}\(x \) / U^{'} \( x \)$ is defined on $\R$ and satisfies
\begin{equation}\label{ARA}
\text{ARA}\( x\) = \frac{\alpha}{\sqrt{\beta^2 + \(x-d\)^2}} \geq 0,\quad x \in \R,
\end{equation}
for given $\alpha \geq 0$ (the \textit{risk aversion parameter}), $\beta > 0$ (the \textit{scale parameter}), and $d \in \R$ (the \textit{threshold wealth}). 
To give the explicit expression of such a utility, 
we derive that there exist constants $c_1 \in \R$ and $c_2 > 0$ such that $U\( x \) = c_1 + c_2 \hat{U}\( x \)$ with
\begin{equation}\label{basic_U}
\hat{U}(x;\alpha,\beta,d) = \left\{
\begin{aligned}
&- \frac{1}{\alpha^2 - 1} \( \(x-d\) + \sqrt{\beta^2 + \(x-d\)^2}\)^{-\alpha} \(\(x-d\)+ \alpha \sqrt{\beta^2 + \(x-d\)^2}\),\quad \alpha \neq 1;\\
&\frac{1}{2} \log \( \(x-d\) + \sqrt{\beta^2 + \(x-d\)^2} \) + \frac{1}{2} \beta^{-2} \(x-d\) \( \sqrt{\beta^2 + \(x-d\)^2} -\(x-d\) \),\quad \alpha = 1,
\end{aligned}
\right.
\end{equation}
where the domain is $\R$ in both cases. 
\end{definition}
\begin{remark}
As a detail in the derivation above, we first solve from Eq. \eqref{ARA} that (up to some constants):
\begin{equation}\label{eq:derivative}
\hat{U}'\(x\) = \(x + \sqrt{\beta^2 + x^2} \)^{-\alpha}, \;\; x \in \R,
\end{equation}
and then obtain $\hat{U}$ in Eq. \eqref{basic_U}.
\end{remark}
\begin{remark}\label{rmk:SAHARA and HARA}
In the above definition, we let $\beta \neq 0$, but actually we can include the case $\beta = 0$. In the latter case, we have
\begin{equation}
\text{ARA}\(x\) = \frac{\alpha}{| x - d |}, \;\; x \neq d,
\end{equation}
which reduces to a HARA utility on the interval of $(d,\infty)$. For better analytical tractability, in the case $\beta = 0$, we let $x \in (d,\infty)$ be the domain of the utility function; see also Assumption \ref{assp:assumption 2} and Remark \ref{rmk:PHARA}. 
\end{remark}
Current studies on the SAHARA utility (\cite{CPV2011} and others) usually assume $d = 0$ for notation simplicity. However, we emphasize the role of the parameter $d$ as it represents different threshold wealth levels; see later Definition \ref{def:PSAHARA utility}. As a result, we can have more general properties of the SAHARA family.

For the fund management model, we suppose that the manager receives a constant proportion of the terminal fund value as the management fee. Moreover, he/she is granted a call option with a fixed strike price that he/she can choose whether or not to exercise at maturity. Hence, the terminal wealth of the manager under the compensation scheme takes the form
\begin{equation}\label{eq:incentive contract}
\Theta \(x\) = w \( x - B_T \)^+ + v x,\quad x \in \R,
\end{equation}
where $B_T$ is the discounted benchmark level, $w > 0$ denotes the number of options and $0 < v < 1$ is the rate of management fee. In practice, we usually have that $0 < v < w < 1$ since $w$ represents the incentive compensation while $v$ represents the regular management fee. This is similar to the model in \cite{HJ2007} and \cite{CHN2019} but excludes the lower-bound liquidation boundary. Let $X_T$ denote the terminal fund value. Then the manager's wealth is $\Theta(X_T)$. We assume that the manager has a SAHARA utility $\hat{U}$. Hence, the utility function of the manager with respect to the terminal value of the fund under the compensation scheme is given by
\begin{equation}\label{eq:simple utility}
U\( X_T \) := \hat{U} \circ \Theta\(X_T\) = \left\{\
\begin{aligned}
&\hat{U} \(vX_T\), && X_T \leq B_T;\\
&\hat{U} \(w\(X_T-B_T\) + vX_T\), && X_T > B_T,
\end{aligned}
\right.
\end{equation}
where $\hat{U}$ is the function defined in \eqref{basic_U}.
\begin{remark}
It is a regular approach to set a lower bound or a liquidation boundary ($X_T \geq 0$ a.s.) in the literature; see, e.g., \cite{C2000} and \cite{BS2014} and later Assumption \ref{assp:assumption 2}. They assume the \textit{Inada condition} on their utility function $U$. That is, $U$ is defined on $(0,\infty)$ and
\begin{equation}\label{eq:inada condition1}
\lim_{x \rightarrow 0} U'\(x\) = \infty,\quad\lim_{x \rightarrow \infty} U'\(x\) = 0,\quad \lim_{x \rightarrow 0} U\(x\) = -\infty,
\quad \lim_{x \rightarrow \infty} U\(x\) = \infty.
\end{equation}
By shifting on the $x$-axis, Eq. \eqref{eq:inada condition1} can be made valid for any real number. 
This assumption, though providing analytical tractability, excludes the cases where the fund value tends extremely low. Here, we do not set such a lower bound. Our setting includes the whole real line as the domain of the wealth level. We can hence deal with more comprehensive scenarios.
\end{remark}

The explicit form of Eq. \eqref{eq:simple utility} is
{\small
\begin{equation}\label{eq:incentive utility}
U\(X_T\) = \left\{
\begin{aligned}
& -\frac{v^{1-\alpha}}{\alpha^2-1}  \[ \(X_T - \frac{d}{v}\) + \sqrt{ \frac{\beta^2}{v^2}+\(X_T - \frac{d}{v}\)^2 } \]^{-\alpha} \[ \(X_T-\frac{d}{v}\) + \alpha \sqrt{ \frac{\beta^2}{v^2} + \(X_T - \frac{d}{v}\)^2 } \], && X_T \leq B_T;\\
& -\frac{\(w+v\)^{1-\alpha}}{\alpha^2-1}  \[ \(X_T - \frac{wB_T+d}{w+v}\) + \sqrt{ \frac{\beta^2}{\(w+v\)^2}+\(X_T - \frac{wB_T+d}{w+v}\)^2 } \]^{-\alpha}\\
&\times \[ \(X_T - \frac{wB_T+d}{w+v}\)+\alpha \sqrt{\frac{\beta^2}{\(w+v\)^2}+\(X_T-\frac{wB_T+d}{w+v}\)^2} \], && X_T > B_T,
\end{aligned}
\right.
\end{equation}}which belongs to the PSAHARA family introduced later in Section \ref{sec:model setting}. This is further illustrated in Figure \ref{fig:incentive}.
\begin{figure}[h!]
\centering
\includegraphics[width=0.3\linewidth]{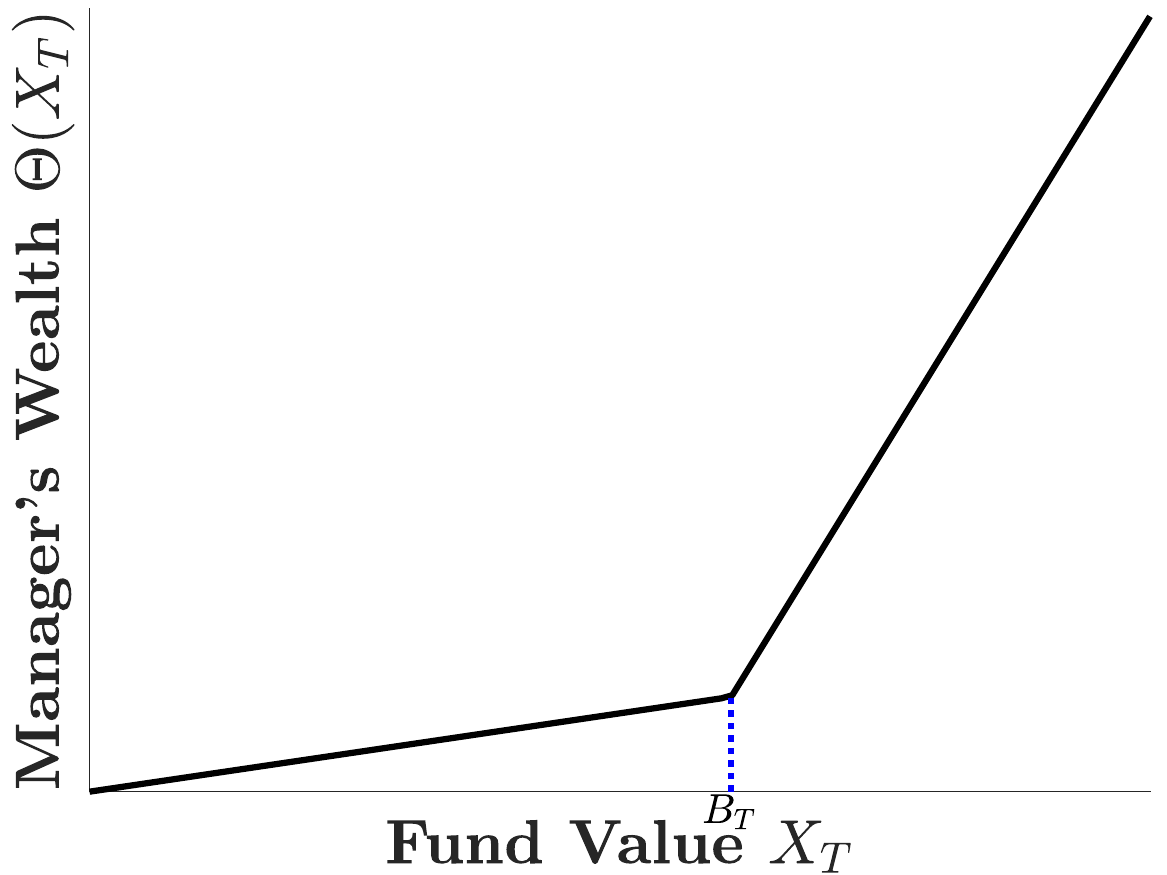} \quad
\includegraphics[width=0.3\linewidth]{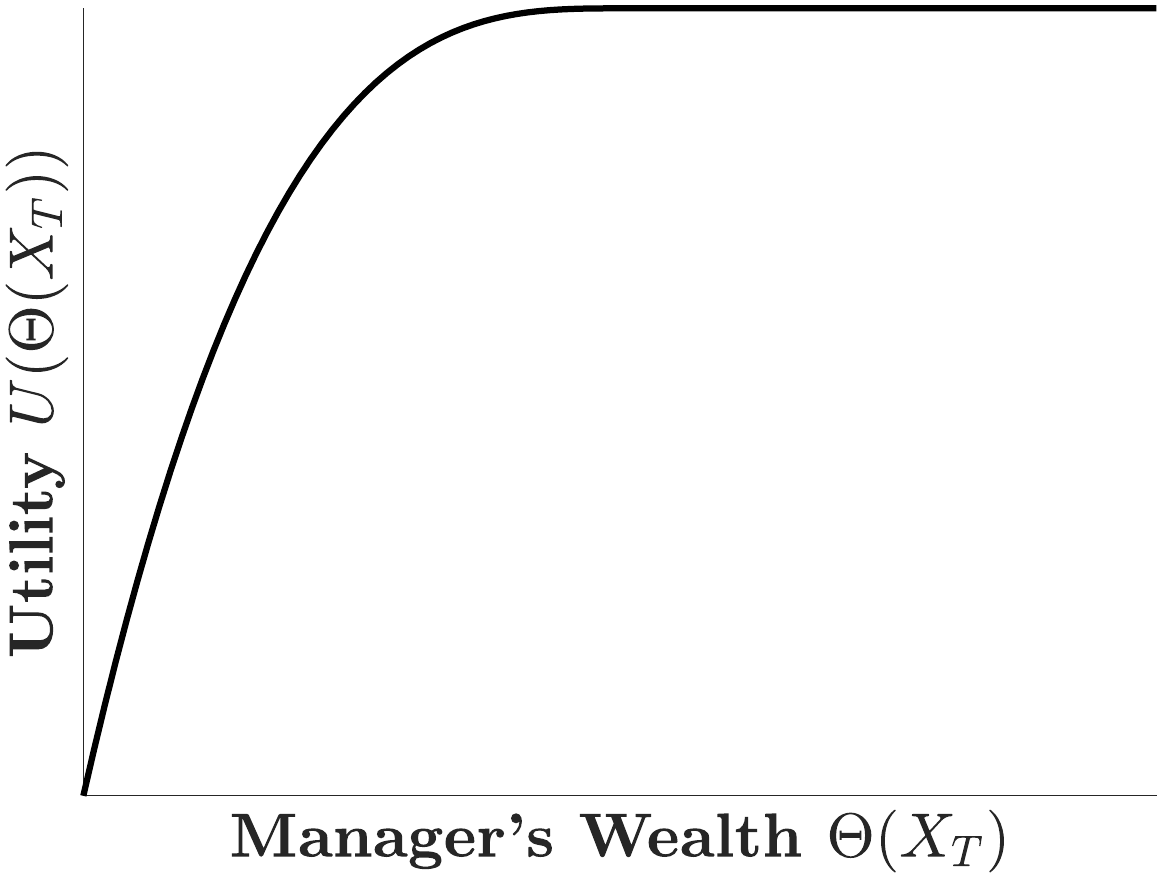} \quad
\includegraphics[width=0.3\linewidth]{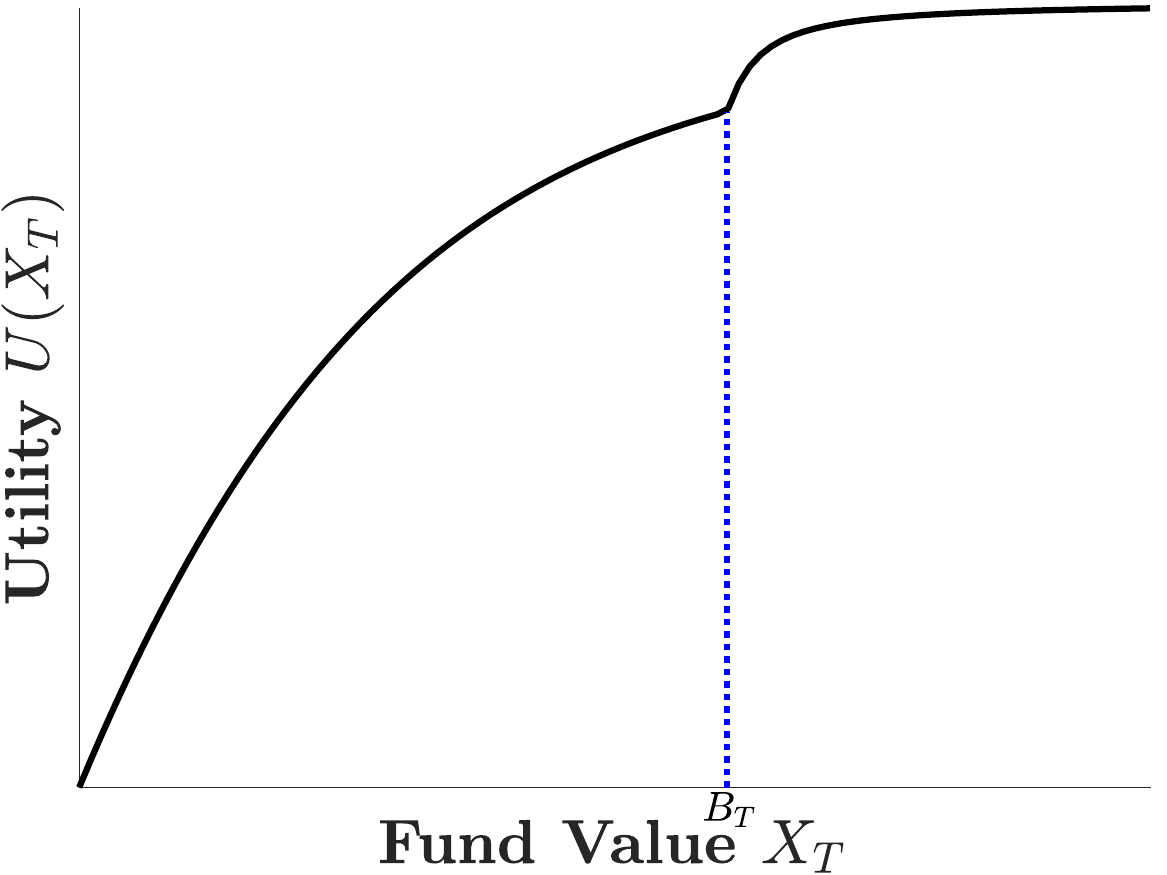}
\caption{The left subfigure is the convex compensation scheme \eqref{eq:incentive contract}. The middle subfigure is the manager's SAHARA utility $\hat{U}$ shown in \eqref{basic_U}. The right subfigure is the composed utility \eqref{eq:incentive utility}. In the parametrization, $\alpha=2,\beta=0.1,d=0,B_T=e^{0.05}$.}
\label{fig:incentive}
\end{figure}

We can further delve into the risk aversion behaviors of the manager in Eq. \eqref{eq:incentive utility}:
{\small
\begin{equation}
\text{ARA} \(X_T\) = \left\{
\begin{aligned}
&\frac{v\alpha}{\sqrt{\beta^2 + \(X_T -d\)^2}}, && X_T \leq B_T;\\
&\frac{\(w+v\)\alpha}{\sqrt{\beta^2 + \(X_T-wB_T-d\)^2}}, && X_T > B_T.
\end{aligned}
\right.
\end{equation}}We see that on both parts of the domain, the manager's risk aversion attitude towards the fund value is decreased (since in most cases $0 < v < w < 1$). Also, the underlying threshold level of the utility is increased by the payoff of the option above the strike price of the incentive. We first set up the market in the following and study the manager's investment behaviors. This motivating example will be revisited in Section \ref{sec:application}.

\section{Model Setting}\label{sec:model setting}
\subsection{Market Model}\label{sec:market model}

We consider a multi-dimensional Black--Scholes model. To solve the explicit solution, all the market parameters are deterministic processes. There are a risk-free asset and risky assets in the market. The risk-free asset has a deterministic return rate and no volatility, while the risky assets have higher expected return rates and positive volatilities. We denote by the filtered probability space $ \left( \Omega, \F_T, \{ \F_t \}_{0 \leq t \leq T}, \p \right)$ the financial market. The filtration $\{ \F _t \}_{0 \leq t \leq T}$ is the one generated by a $q$-dimensional standard independent Brownian motion $\{ \mathbf{W}_t \}_{0 \leq t \leq T} = \{\(W_{1,t}, \ldots, W_{q,t}\)^\intercal \}_{0 \leq t \leq T}$ and further augmented by all $\p$-null sets. Let the deterministic process $\{ r_t \}_{0\leq t \leq T}$ denote the risk-free rate. The risk-free asset $\{S_{0,t}\}_{0 \leq t \leq T}$ satisfies
\begin{equation}\label{Bond}
\d S_{0,t} = r_t S_{0,t} \d t,\quad 0 \leq t \leq T.
\end{equation}
For the $m$ risky assets, we denote the return rate by a vector $\pmb{\mu}_t := \(\mu_{1,t},\ldots, \mu_{m,t} \)^\intercal $ and the volatility by an $m \times q$ matrix $\pmb{\sigma}_t$. We assume that $\pmb{\sigma}_t \pmb{\sigma}_t^\intercal$ is positive definite for all $t \in [0,T]$, and hence is invertible.

In Black--Scholes models, market completeness means that every risky asset in the market can be replicated by a self-financing trading strategy, which is equivalent to the assumption that every risk can be hedged (i.e., $m=q$). As a study on the incomplete market, we consider that the risks may not be totally hedged, i.e., $m \leq q$. We assume that $\mu_{i,t} > r_t$ for $i = 1, \ldots, m, \; t \in \[0, T\]$. The evolution of the $i$-th risky asset follows the geometric Brownian motion which satisfies the stochastic differential equation:
\begin{equation}\label{Stock}
\d S_{i,t} = \mu_{i,t} S_{i,t} \d t + S_{i,t} \pmb{\sigma}_{i,t}^\intercal \d \mathbf{W}_t, \quad i = 1, \ldots, m,
\end{equation}
where $\pmb{\sigma}_{i,t}$ denotes the $i$th row of $\pmb{\sigma}_t$. Letting $\mathbf{1}_m := (1, \ldots, 1)^\intercal \in \R^m$, we define
\begin{equation}\label{eq:chosen theta}
\pmb{\theta}_t := \pmb{\sigma}_t^\intercal \( \pmb{\sigma}_t \pmb{\sigma}_t^\intercal \)^{-1} \( \pmb{\mu}_t - r_t \mathbf{1}_m \),
\end{equation}
which represents the vector of the market price of risk. 
\begin{remark}
In the traditional incomplete Black--Scholes model, the market price of risk $\{\Tilde{\pmb{\theta}}_t \}_{0\leq t\leq T}$ can be any solution of the linear system:
\begin{equation}\label{eq:price of risk}
\pmb{\sigma}_t \Tilde{\pmb{\theta}}_t = \pmb{\mu}_t - r_t \mathbf{1}_m.
\end{equation}
Hence, if $m < q$, $\Tilde{\pmb{\theta}}_t$ is not unique. We choose Eq. \eqref{eq:chosen theta}, one of the solutions of Eq. \eqref{eq:price of risk}, to represent the price of risk. Later in Section \ref{sec:technical discussions}, we will show that this simplification does not influence the validity of our main theorems.  
\end{remark}
We proceed to give a standing assumption on market coefficients. The following is supposed to hold in the rest of the paper.
\begin{assumption}
The deterministic processes $\{\pmb{\mu}_t \}_{0\leq t\leq T}$, $\{r_t\}_{0\leq t\leq T}$ and $\{\Tilde{\pmb{\theta}}_t\}_{0\leq t\leq T}$ in \eqref{eq:price of risk} satisfy
\begin{equation}
    \sup_{t \in [0,T]} \| \pmb{\mu}_t \|_2 < \infty,\quad \sup_{t \in [0,T]} r_t  < \infty,\quad \sup_{t \in [0,T]} \| \Tilde{\pmb{\theta}}_t \|_2 < \infty,
\end{equation}
and
\begin{equation}\label{eq:Novikov condition}
    \exp \left\{ \frac12 \int_0^T \| \pmb{\theta}_t \|_2^2 \, \d t \right\} < \infty,
\end{equation}
where $\| \pmb{\theta}_t \|_2 := \(\sum_{i = 1}^q \theta_{i,t}^2\)^{\frac12} $ represents the $L^2$ norm of a vector.
Moreover, for the matrix--valued process $\{\pmb{\sigma}_t \}_{0\leq t\leq T}$, let $\lambda_t^{\max} $ denote the largest eigenvalue of $\pmb{\sigma}_t\pmb{\sigma}_t^\intercal$. Then the process $\{\lambda_t^{\max}\}_{0\leq t\leq T} $ satisfies
\begin{equation}
    \sup_{t\in [0,T]} \lambda_t^{\max}< \infty.
\end{equation}
\end{assumption}
The assumption above is reasonable since the coefficients are deterministic processes. 
Next, we define one pricing kernel process $\{ \xi_t \}_{0 \leq t \leq T}$ as follows:
\begin{equation}\label{eq:pricing kernel}
\xi_t := \exp \left\{ - \int_0^t \( r_s + \frac12 \| \pmb{\theta}_s \|_2^2 \)\, \d s - \int_0^t \pmb{\theta}_s^\intercal \, \d \mathbf{W}_s\right\}, \quad 0 \leq t \leq T.
\end{equation}
We denote by $\{\pi_{i,t}\}_{0 \leq t \leq T}$ the amount of money invested in $i$th risky asset $S_{i}$ at time $t$. The wealth process $\{X_t\}_{0 \leq t \leq T}$ is uniquely determined by the investment process $\{\pmb{\pi}_t = \(\pi_{1,t}, \ldots, \pi_{m,t} \) \}_ {0 \leq t \leq T}$ and an initial value $x_0 \in \R$:
\begin{equation}\label{wealthprocess}
\d X_t = \( r_t X_t + \pmb{\pi}_t^\intercal \(  \pmb{\mu}_t - r_t \mathbf{1}_m\)\) \d t + \pmb{\pi}_t^\intercal \pmb{\sigma}_t \d \mathbf{W}_t, \quad X_0 = x_0.
\end{equation}
To define the admissible set of controls, we introduce the following assumptions on $\{\pmb{\pi}_t \}_{0\leq t\leq T}$. We will suppose either of Assumptions \ref{assp:assumption 1} or \ref{assp:assumption 2} holds in the rest of the paper.
\begin{assumption}\label{assp:assumption 1}
There exists a uniform $\epsilon > 0$ such that
\begin{equation}
\E \[\int_0^T \|\pmb{\pi}_t\|_2^{2+\epsilon} \,\d t\] < \infty \quad \text{for all}\; \{\pmb{\pi}_t \}_{0\leq t\leq T}.
\end{equation}
\end{assumption}
\begin{assumption}\label{assp:assumption 2}
There exists a uniform constant $C \in \R$ such that the wealth process $\{X_t\}_{0\leq t\leq T}$ controlled by any $\{\pmb{\pi}_t \}_{0\leq t\leq T}$ satisfies
\begin{equation}
    X_t \geq C\; \text{a.s.},\quad 0\leq t\leq T.
\end{equation}
\end{assumption}
\begin{remark}
As we will show later in Section \ref{sec:technical discussions}, either of the two assumptions guarantees the validity of martingale and duality method. Assumption \ref{assp:assumption 2} is widely adopted in current studies and reflected in utility functions; see, e.g., \cite{KLS1987}, where $C$ is taken as $0$. We present Assumption \ref{assp:assumption 1} to show a rather explicit condition on $\{\pmb{\pi}_t\}_{0\leq t\leq T}$ for the martingale and duality method to work.
\end{remark}
We will discuss later in Section \ref{sec:technical discussions} about the details of these two assumptions, respectively.
Now, we define the admissible set $\mathcal{V}$ of controls. 
\begin{equation}
\begin{aligned}
\mathcal{V} := \left\{\pmb{\pi} : [0,T] \times \Omega \rightarrow \R^m \right. | & \pmb{\pi} \;\text{is}\; \{ \mathcal{F}_t \}_{0 \leq t \leq T}-\text{progressively measurable and satisfies} \\
& \left. \text{either of Assumption \ref{assp:assumption 1} and Assumption \ref{assp:assumption 2} } \right\}.
\end{aligned}
\end{equation}
A wealth process is called admissible if it is controlled by an admissible control.
The decision maker conducts portfolio selection by solving the expected utility maximization problem:
\begin{equation}\label{eq:main problem}
    \max_{\pmb{\pi} \in \mathcal{V}} \E \[ U \( X_T\)\],
\end{equation}
where $U$ is the composed utility function. 

\subsection{PSAHARA Utility}\label{sec:PSAHARA utility}
Next, we define the piecewise symmetric asymptotic hyperbolic absolute risk aversion (PSAHARA) utility function. Namely, it can be viewed as the SAHARA utility function on each part of its domain.
\begin{definition}[PSAHARA utility]\label{def:PSAHARA utility}
Define the function $\Tilde{U} : \R \rightarrow \R$ (with risk aversion parameter $\alpha \geq 0$, scale parameter $\beta > 0$, threshold level $d \in \R$, utility value $u \in \R$, slope $\gamma > 0$) as 
\begin{equation}\label{U_tilde}
    \Tilde{U} \( x; \alpha, \beta, d, \gamma, u \) := \gamma \hat{U}\( x; \alpha, \beta, d \) + u.
\end{equation}
A function $U: \R \rightarrow \R$ is a piecewise SAHARA utilitiy if and only if there exists a partition $\{ a_k \}_{k=0}^{n+1}$ and a family of parameter tuples $\{ \( \alpha_k, \beta_k, d_k,u_k \) \}_{k=0}^{n}$ such that
\begin{enumerate}
\item[(i)] $n \geq 0, a_1 < a_2 < \cdots < a_n ; a_1, \ldots, a_n \in \R , a_0 = -\infty, a_{n+1} = \infty;$
\item[(ii)] $U$ is increasing and continuous on $\R$;
\item[(iii)]
\begin{enumerate}
    \item[(a)] If $n = 0$, then $U \(x\) = \gamma \hat{U}\( x; \alpha, \beta, d\) + u$ for any $x \in \R$;
    \item[(b)] If $n \geq 1$, for $k = 0, U\(x\) = \Tilde{U}\(x; \alpha_0, \beta_0, d_0, \gamma_0, u_0\)$ for any $x \in (a_0, a_1)$; for any $ k \in \{ 1,2,\ldots, n\}, U\(x\) = \Tilde{U}\( x; \alpha_k, \beta_k, d_k, \gamma_k, u_k \) $ for any $x \in \( a_k, a_{k+1} \)$.
\end{enumerate}
\end{enumerate}
\end{definition}
For simplicity, we use the notation: $\gamma_k^+ = U'(a_k^+), \gamma_k^- = U'(a_k^-)$, where $\gamma_{n+1}^- = 0$ in the following context.
\begin{remark}
If $\beta_k = 0$ for $k = {0, \ldots, n}$, the PSAHARA utility reduces to the PHARA utility defined in \cite{LLMV2024}.
\end{remark}
Clearly, the PSAHARA utility function can be non-concave. Hence, we introduce the concept of the concave envelope. 
\begin{definition}[Concave envelope]
Let $\mathcal{D} \subseteq \R$ be a convex set. Denote a continuous function by $U: \mathcal{D} \rightarrow \R$, where the domain of $U$ is denoted by dom $U = \D$. The concave envelope of $U$ (denoted by $U^{**}$) is defined as the smallest continuous concave function larger than $U$. That is, for $x \in \D$,
\begin{equation}\label{envelope_def}
    U^{**}\(x\) := \inf \{h\(x\):\text{ $h$ maps $\D$ to $\R$},\text{$h$ is a concave and continuous function on $\D$ and $h \geq U$}\}.
\end{equation}
\end{definition}

The concave envelope plays an important role in the portfolio choice problem with non-concave utilities. We rigorously show in the proof of Theorem \ref{thm:optimal wealth} that if $\xi_T$ has a continuous distribution, Problem \eqref{eq:main problem} has the same optimal solution as the following problem with the utility $U$ replaced by the concave envelope $U^{**}$:
\begin{equation}\label{problem_enve}
    \max_{\pmb{\pi} \in \mathcal{V}} \E \[ U^{**} \(X_T\)\]
\end{equation}
We give two propositions in the following to show the generality of the PSAHARA utility family. 
\begin{proposition}\label{prop:generality}
If $U\(\cdot\)$ is a PSAHARA utility function, and $h\(\cdot\)$ is an increasing continuous piecewise linear function, then $U \circ h\(\cdot\)$ is also a PSAHARA utility function. 
\end{proposition}
\begin{proposition}\label{prop:enve}
If $U\(\cdot\)$ is a PSAHARA utility function, then $U^{**}\(\cdot\)$ is also a PSAHARA utility function. Further, even if $U\(\cdot\)$ is not a PSAHARA utility function, $U^{**}\(\cdot\)$ may be a PSAHARA utility function.
\end{proposition}
These propositions are especially useful in dealing with option incentive schemes in hedge fund management. For a fund manager with SAHARA preference $U$ and an incentive contract $h$ (motivating example in Section \ref{sec:motivation}), his/her utility after composition is $U \circ h$, which is a PSAHARA utility by Proposition \ref{prop:generality}. Then the concave envelope $(U \circ h)^{**}$ is also a PSAHARA utility by Proposition \ref{prop:enve}. We can thus apply Theorem \ref{optimal portfolio} below to get the manager's optimal portfolio.
\begin{example}\label{eg:main example}
This example further shows the generality of the PSAHARA utility family. We consider a non-monotone and discontinuous utility function:
\begin{equation}\label{eq:main example}
U\(x\) := \left\{
\begin{aligned}
&-\frac{\gamma_0}{\alpha_0^2-1}\( \(x-d_0\) + \sqrt{\beta^2 +\(x-d_0\)^2 } \)^{-\alpha_0}\(\(x-d_0\) + \alpha_0 \sqrt{\beta^2 +\(x-d_0\)^2 }\),&& x < l_1;\\
& -20 \(l_2 -x\)^{1-\alpha_0} + b_1, && l_1 \leq x < l_2;\\
& -\frac{\gamma_3}{\alpha_3^2-1} \(\(x-d_3\)+\sqrt{\beta^2 + \( x-d_3\)^2} \)^{-\alpha_3}\(\(x-d_3\)+\alpha_3\sqrt{\beta^2 + \( x-d_3\)^2}\) + b_2, && l_2 \leq x < l_3;\\
&b_3, && l_3 \leq x < l_4;\\
& -\frac{\gamma_5}{\alpha_5^2-1} \(\(x-d_5\)+\sqrt{\beta^2 + \( x-d_5\)^2} \)^{-\alpha_5}\(\(x-d_5\)+\alpha_5\sqrt{\beta^2 + \( x-d_5\)^2}\) + b_4, && x \geq l_4,\\ 
\end{aligned}
\right.
\end{equation}
where we let $l_1 = -6, l_2 = -4.5, l_3 = -1, l_4 = 2, \beta = 1, \alpha_0 = 1.7, \alpha_3 = 2.2, \alpha_5 = 1.2, d_0 = 3, d_3 = 1, d_5 = 6, \gamma_0 = 2, \gamma_3 = 1.5, \gamma_5 = 7$. $b_1, b_2, b_3, b_4$ are suitable real numbers to make the function continuous on $x \geq l_1$.
This utility function is not a PSAHARA utility since there exist decreasing segments and jumps. However, as shown in Figure \ref{fig:main example}, the concave envelope of \eqref{eq:main example} is a PSAHARA utility.
\begin{figure}[t]
    \begin{minipage}{0.5\textwidth}
    \centering
    \includegraphics[width=\textwidth]{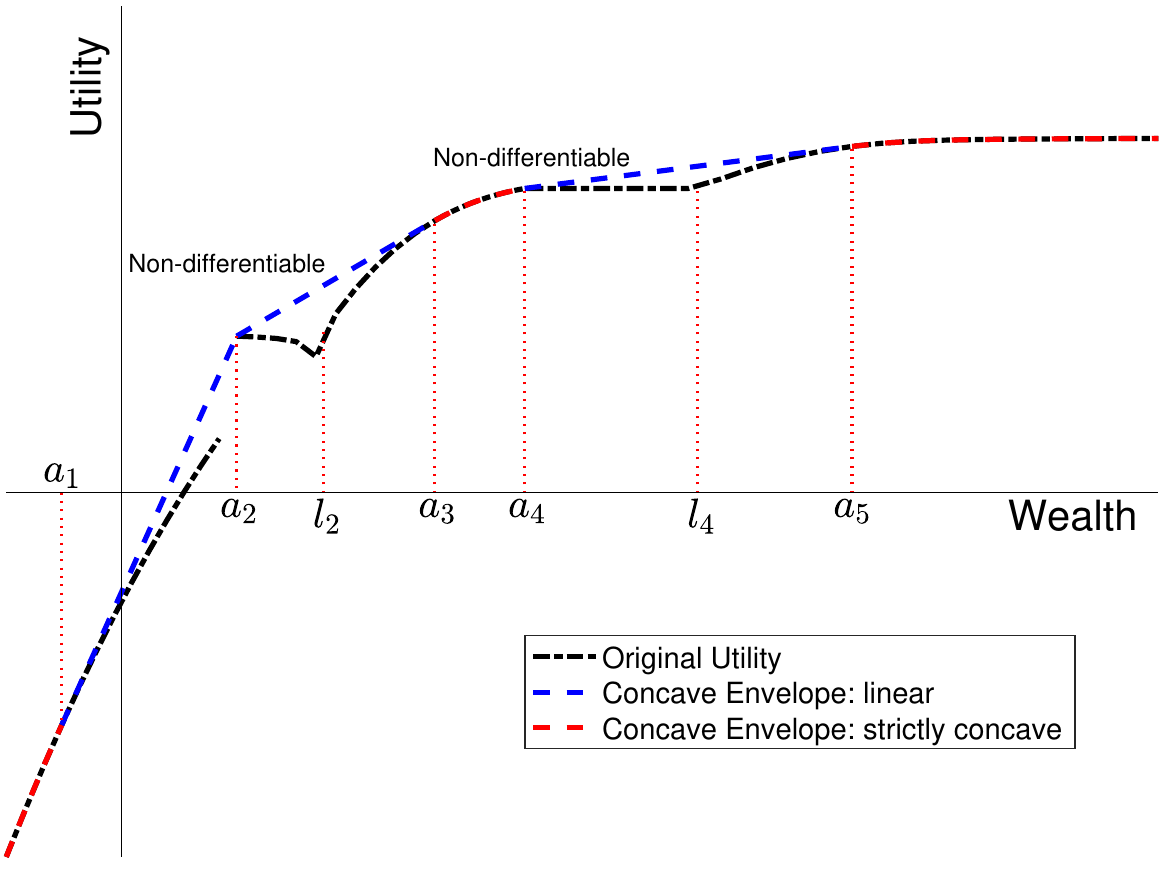}
    \end{minipage}
    \hspace{1em plus 1fill}
    \begin{minipage}{0.48\textwidth}
    \begin{equation*}
    U^{**}\(x\) := \left\{
    \begin{aligned}
    & \gamma_0 \hat{U}\(x; \alpha_0, \beta, d_0\) + u_0, && x < a_1;\\
    & \gamma_0 \hat{U}'\(a_1; \alpha_0, \beta, d_0\) \(x - a_1\)\\
    & \quad + \hat{U}\(a_1; \alpha_0, \beta, d_0\), && a_1 \leq x < a_2;\\
    & \gamma_3 \hat{U}'\(a_3; \alpha_3, \beta, d_3\) \(x - a_3\)\\
    & \quad + \hat{U}\(a_3; \alpha_3, \beta, d_3\), && a_2 \leq x < a_3;\\
    & \gamma_3 \hat{U}\(x; \alpha_3, \beta, d_3\) + u_3, && a_3 \leq x < a_4;\\
    & \gamma_5 \hat{U}'\(a_5; \alpha_5, \beta, d_5\) \(x - a_5\)\\
    & \quad + \hat{U}\(a_5; \alpha_5, \beta, d_5\), && a_4 \leq x < a_5;\\
    & \gamma_5 \hat{U}\(x; \alpha_5, \beta, d_5 \) + u_5, && x \geq a_5.
    \end{aligned}
    \right.
    \end{equation*}
    \end{minipage}
    \caption{The original utility \eqref{eq:main example} and the concave envelope. In the figure, the red parts are the segments where the concave envelope coincides with the original utility, while the blue straight lines are the segments where the concave envelope does not coincide with the original utility. $\hat{U}$ is the utility function defined in \eqref{basic_U} and $\hat{U}'$ is the derivative of $\hat{U}$ shown in \eqref{eq:derivative}. $a_1, a_3$, $a_5$ are the tangent points between the linear segments and the strictly concave parts. $(l_1, l_3) = (a_2, a_4)$. $u_0, u_3, u_5$ are the correction constants to make the function continuous. As labeled in the figure, $U^{**}$ is non-differentiable at $a_2$ and $a_4$.}
    \label{fig:main example}
\end{figure}

\end{example}

\section{A Unified Formula of the Optimal Portfolio}
\subsection{Optimal Wealth and Optimal Portfolio}\label{sec:optimal portfolio}
We proceed to show the optimal terminal wealth, the optimal wealth process and the optimal control of PSAHARA utilities for Problem \eqref{eq:main problem}. We adopt the martingale and duality method and the concavification technique. The proofs are included in Section \ref{sec:technical discussions} and Appendix \ref{proofs}.
\begin{theorem}\label{thm:optimal wealth}
For a given utility function $U$, suppose its concave envelope $U^{**}$ takes the form in Definition \ref{def:PSAHARA utility} and $\alpha_k \in [ 0, \infty )$ for each $k \in \{ 0, 1, \ldots, n\}$. For Problem \eqref{eq:main problem},\\ 
(1) the optimal terminal wealth is given by
\begin{equation}\label{opt_terminal_wealth}
    X_T^* = \sum_{k=1}^n a_k \mathds{1}_{\left\{ y^* \xi_T \in \( \gamma_k^+, \gamma_k^-\) \right\}}
     +\sum_{k=0}^n \left\{ \left( d_k + \frac{1}{2} \( \( \frac{\gamma_k}{y^* \xi_T} \)^{\frac{1}{\alpha_k}} - \beta_k^2 \( \frac{\gamma_k}{y^* \xi_T} \)^{-\frac{1}{\alpha_k}} \) \right) \mathds{1}_{\left\{y^*\xi_T \in \( \gamma_{k+1}^-, \gamma_k^+\) \right\} } \right\}, \; a.s.,
\end{equation}
where $y^*$ is a unique positive number that satisfies
\begin{equation}\label{eq:Lagrange multiplier}
\E \[ \xi_T X_T^*\] = x_0;
\end{equation}
(2) the optimal wealth at time $t \in [ 0, T )$ is given by
\begin{equation}\label{opt_t_wealth}
\begin{aligned}
    X_t^* &:= X_t^D + X_t^B + X_t^R + X_t^{\bar{R}}\\
    &= \sum_{k=0}^n \( X_{t,k}^D + X_{t,k}^B + X_{t,k}^R + X_{t,k}^{\bar{R}} \),
\end{aligned}
\end{equation}
where
{\small
\begin{align}\label{part_t_wealth}
&X_{t,k}^D = && \left\{ 
\begin{aligned}
& e^{-\int_t^T r_s \, \d s} a_k \[ \Phi \( g_0\( \frac{\gamma_k^+}{y^* \xi_t} \) \) - \Phi \( g_0\( \frac{\gamma_k^-}{y^* \xi_t} \) \) \], && k \neq 0;\\
& 0, && k = 0,
\end{aligned}
\right.\nonumber\\
&X_{t,k}^B = && e^{-\int_t^T r_s \, \d s} d_k \[ \Phi \( g_0\( \frac{\gamma_{k+1}^-}{y^* \xi_t} \) \) - \Phi \( g_0\( \frac{\gamma_k^+}{y^* \xi_t} \) \) \] \mathds{1}_{\{\alpha_k \neq 0 \} }, \nonumber\\
&X_{t,k}^R = && e^{\(-1+\frac{1}{\alpha_k}\)\int_t^T \(r_s + \frac{1}{2\alpha_k}\| \pmb{\theta}_s \|_2^2 \, \)\d s  } \frac{1}{2} \( \frac{\gamma_k}{y^* \xi_t} \)^{\frac{1}{\alpha_k}}\( \Phi\( g_{1,k} \( \frac{\gamma_{k+1}^-}{y^* \xi_t} \) \) - \Phi\( g_{1,k} \( \frac{\gamma_k^+}{y^* \xi_t} \) \) \) \mathds{1}_{\{\alpha_k \neq 0\}}, \nonumber\\
&X_{t,k}^{\bar{R}} = &&e^{\(-1-\frac{1}{\alpha_k}\)\int_t^T \(r_s - \frac{1}{2\alpha_k}\| \pmb{\theta}_s \|_2^2 \, \)\d s  } \( -\frac{1}{2}\) \beta_k^2 \( \frac{\gamma_k}{y^* \xi_t} \)^{-\frac{1}{\alpha_k}}\( \Phi\( g_{2,k} \( \frac{\gamma_{k+1}^-}{y^* \xi_t} \) \) - \Phi\( g_{2,k} \( \frac{\gamma_k^+}{y^* \xi_t} \) \) \) \mathds{1}_{\{\alpha_k \neq 0\}},
\end{align}}and the functions $g_0 ( \cdot),\; g_{1,k}( \cdot),\; g_{2,k}(\cdot)$ are given by
{\small
\begin{equation}\label{g0}
   g_0 \(z\) := -\frac{1}{\sqrt{ \int_t^T \| \pmb{\theta}_s \|_2^2 \, \d s }} \( \log\(z\) + \int_t^T \(r_s - \frac{1}{2} \| \pmb{\theta}_s\|_2^2 \) \, \d s \), \quad z > 0, 
\end{equation}}
{\small
\begin{equation}\label{g1g2}
    g_{1,k}\(z\) := g_0\(z\) -\frac{1}{\alpha_k} \sqrt{\int_t^T \| \pmb{\theta}_s \|_2^2 \, \d s },\quad g_{2,k}\(z\) := g_0\(z\) + \frac{1}{\alpha_k} \sqrt{\int_t^T \| \pmb{\theta}_s \|_2^2 \, \d s },\quad z > 0, \quad k \in \{0, 1,\ldots, n\}.
\end{equation}}
\end{theorem}
The optimal wealth at time $t$ consists of four components, each representing an aspect of the manager's investment behavior. A further illustration is given below in Theorem \ref{optimal portfolio}.
\begin{itemize}
\item The terms $X_{t,k}^R$ and $X_{t,k}^{\bar{R}}$ arise from the symmetric nature of SAHARA utility functions, representing the investor's propensity towards the risk-seeking behavior. Specifically, $X_{t,k}^R$ encourages an increase to the risky investment when the wealth exceeds the threshold level, and hence is called the positive main term. Conversely, $X_{t,k}^{\bar{R}}$ promotes an increase to the risky investment when wealth is below the threshold, and it is hence referred to as the negative main term. These two terms are the main driver of $X_t^*$, as illustrated in the later Section \ref{asymptotic analysis} of asymptotic analysis.
\item The term $X_{t,k}^D$ is associated with the non-differentiable point $a_k$ and hence is called the first-order risk aversion term. This component plays a pivotal role leading to the manager's risk aversion around non-differentiable points, as evidenced by the term $\pmb{\pi}_t^{(3)}$ later in Theorem \ref{optimal portfolio}. It leads to the reduction in risky investment due to the first-order risk aversion.
\item The term $X_{t,k}^B$ is the loss aversion term as it directly comes from the threshold level $d_k$. This term influences the optimal portfolio through its contribution to the term $\pmb{\pi}_t^{(4)}$, encapsulating the loss aversion caused by the threshold levels.
\end{itemize}

In the following theorem, we show the optimal control of Problem \eqref{problem_enve}, which is also the unified formula of the optimal portfolio for PSAHARA utilities. The proof is included in Section \ref{sec:technical discussions} and Appendix \ref{proofs}.

\begin{theorem}\label{optimal portfolio}
Suppose that the concave envelope $U^{**}$ has the form in Definition \ref{def:PSAHARA utility} and $\alpha_k \in [ 0, \infty)$ for each $k \in \{0, 1, \ldots, n\}$. For Problem \eqref{eq:main problem}, the optimal portfolio at time $t \in [ 0, T)$ is 
\begin{equation}\label{pi_t}
\pmb{\pi}_t^* = \pmb{\pi}_{t}^{(1)} + \pmb{\pi}_{t}^{(2)} +  \pmb{\pi}_{t}^{(3)} + \pmb{\pi}_{t}^{(4)},
\end{equation}
where
{\small
\begin{align}\label{parts_pi_t}
\pmb{\pi}_{t}^{(1)} &= \(\pmb{\sigma}_t\pmb{\sigma}_t^\intercal\)^{-1} \(\pmb{\mu}_t - r_t \mathbf{1}_m \) \sum_{k=0}^n \frac{1}{\alpha_k} \sqrt{ \(X_{t,k}^R + X_{t,k}^{\bar{R}} \)^2 + b_{t,k} }\times \mathds{1}_{\left\{ \alpha_k \neq 0 \right\}}, \nonumber\\
\pmb{\pi}_{t}^{(2)} &= -\frac{\(\pmb{\sigma}_t\pmb{\sigma}_t^\intercal\)^{-1} \(\pmb{\mu}_t - r_t \mathbf{1}_m \) }{\sqrt{\int_t^T \| \pmb{\theta}_s \|_2^2 \, \d s}} \sum_{k=0}^n \left\{ X_{t,k}^R \frac{ \Phi' \(g_{1,k} \(\frac{\gamma_{k+1}^-}{y^* \xi_t} \) \) - \Phi'\(g_{1,k} \(\frac{\gamma_k^+}{y^* \xi_t} \)\)}{ \Phi \(g_{1,k} \(\frac{\gamma_{k+1}^-}{y^* \xi_t} \) \) - \Phi\(g_{1,k} \(\frac{\gamma_k^+}{y^* \xi_t} \)\)} \right.\nonumber \\
&\quad\quad\quad\quad\quad\quad\quad\quad\quad\quad\quad\quad +\left. X_{t,k}^{\bar{R}} \frac{ \Phi' \(g_{2,k} \(\frac{\gamma_{k+1}^-}{y^* \xi_t} \) \) - \Phi'\(g_{2,k} \(\frac{\gamma_k^+}{y^* \xi_t} \)\)}{ \Phi \(g_{2,k} \(\frac{\gamma_{k+1}^-}{y^* \xi_t} \) \) - \Phi\(g_{2,k} \(\frac{\gamma_k^+}{y^* \xi_t} \)\)} \right\}\times \mathds{1}_{\left\{ \alpha_k \neq 0 \right\}},\nonumber\\
\pmb{\pi}_t^{(3)} &= -\frac{\(\pmb{\sigma}_t\pmb{\sigma}_t^\intercal\)^{-1} \(\pmb{\mu}_t - r_t \mathbf{1}_m \) }{\sqrt{\int_t^T \| \pmb{\theta}_s \|_2^2 \, \d s}} e^{-\int_t^T r_s \, \d s } \sum_{k=1}^{n} \left\{ a_k \[ \Phi' \( g_0 \( \frac{\gamma_k^+}{y^* \xi_t} \) \) - \Phi' \( g_0 \( \frac{\gamma_k^-}{y^* \xi_t} \) \) \] \right\},\nonumber\\
\pmb{\pi}_t^{(4)} &= -\frac{\(\pmb{\sigma}_t\pmb{\sigma}_t^\intercal\)^{-1} \(\pmb{\mu}_t - r_t \mathbf{1}_m \) }{\sqrt{\int_t^T \| \pmb{\theta}_s \|_2^2 \, \d s}}e^{-\int_t^T r_s \, \d s } \sum_{k=0}^{n} \left\{ d_k \[ \Phi' \( g_0 \( \frac{\gamma_{k+1}^-}{y^* \xi_t} \) \) - \Phi' \( g_0 \( \frac{\gamma_k^+}{y^* \xi_t} \) \) \] \right\},
\end{align}}and
{\small
\begin{equation}\label{b_tk}
b_{t,k} := \beta_k^2 e^{2\int_t^T \(-r_s + \frac{1}{2\alpha_k^2} \| \pmb{\theta}_s \|_2^2\) \, \d s } \(\Phi \(g_{1,k} \(\frac{\gamma_{k+1}^-}{y^* \xi_t} \) \) - \Phi\(g_{1,k} \(\frac{\gamma_k^+}{y^* \xi_t} \)\) \)\(\Phi \(g_{2,k} \(\frac{\gamma_{k+1}^-}{y^* \xi_t} \) \) - \Phi\(g_{2,k} \(\frac{\gamma_k^+}{y^* \xi_t} \)\) \).
\end{equation}}
\end{theorem}
\begin{remark}
For the simplest case where $n = 0$ and $ \gamma = 1$, the optimal portfolio is reduced to
\begin{equation}
\pmb{\pi}_t^* = \frac{\( \pmb{\sigma}_t \pmb{\sigma}_t^\intercal \)^{-1} \( \pmb{\mu}_t - r_t \mathbf{1}_m \)}{\alpha} \sqrt{ \(X_t^* - de^{-\int_t^T r_s \, \d s} \)^2 + b_t^2 },
\end{equation}
where
$b_t := \beta e^{ -\int_t^T \( r_s - \frac{\| \pmb{\theta}_s \|_2^2}{\alpha^2} \)\, \d s }$. This coincides with Theorem 3.2 in \cite{CPV2011}.
\end{remark}
The optimal portfolio structure for PSAHARA utilities is divided into four terms, each characterizing both economic and mathematical implications.

We name the term $\pmb{\pi}_t^{(1)}$ the asymptotic Merton term. As illustrated later in Section \ref{asymptotic analysis}, this term drives $\pmb{\pi}_t^*/X_t^*$ to the deterministic proportion $\frac{\pmb{\sigma}_t^\intercal \( \pmb{\sigma}_t \pmb{\sigma}_t^\intercal \)^{-1} \( \pmb{\mu}_t - r_t \mathbf{1}_m \) }{\alpha_n}$ as $\xi_t$ approaches $0$ and $-\frac{\pmb{\sigma}_t^\intercal \( \pmb{\sigma}_t \pmb{\sigma}_t^\intercal \)^{-1} \( \pmb{\mu}_t - r_t \mathbf{1}_m \) }{\alpha_0}$ as $\xi_t$ approaches $\infty$, corresponding to the famous Merton proportion (\cite{M1969}). We claim that $\pmb{\pi}_t^{(1)}$ serves as the main driver behind the dynamics of the optimal portfolio.

The second term $\pmb{\pi}_t^{(2)}$ is identified as the risk adjustment term. It leads to risk-seeking behaviors at high wealth levels and results in risk-aversion at low wealth levels. Containing the scale parameter $\beta$ in $X_{t,k}^R$, this term shows a distinct feature that it adjusts the optimal portfolio corresponding to the total wealth scale. 

We name $\pmb{\pi}_t^{(3)}$ as the first-order risk aversion term. It emerges from the non-differentiable points of the utility function. As indicated in \cite{SS1997}, non-differentiable points incur a ``first-order'' risk premium. Here, our finding coincides with the term $\pmb{\pi}_t^{(4)} $ in Theorem 1 in \cite{LLMV2024}, where it is shown that non-differentiable points lead to a decrease in the optimal portfolio.

The term $\pmb{\pi}_t^{(4)}$ is the loss aversion term. It is the weighted sum of the threshold levels. Since it emerges from the threshold level $d$, it vanishes if $d = 0$. It decreases the optimal portfolio in a less amount compared to $\pmb{\pi}_t^{(3)}$. Both terms $\pmb{\pi}_t^{(3)}$ and $\pmb{\pi}_t^{(4)}$ are local corrections based on non-differentiability and threshold levels of the PSAHARA utility.
\begin{remark}\label{rmk:PHARA}
If we impose the constraint $x > d$ when $\beta = 0$ as in Remark \ref{rmk:SAHARA and HARA}, we have a well-defined HARA utility as we set $\beta = 0$. In this case, the optimal portfolio given by Theorem \ref{optimal portfolio} coincides with the formula given by Theorem 1 in \cite{LLMV2024}. Also, the threshold levels highly correspond to the benchmark levels in that study.
\end{remark}
\subsection{Technical Discussions}\label{sec:technical discussions}
This subsection discusses technical details about how Assumptions \ref{assp:assumption 1} and \ref{assp:assumption 2} work in proving Theorems \ref{thm:optimal wealth}-\ref{optimal portfolio}.
Before presenting the proof, we first claim that one difficulty to show  Theorem \ref{thm:optimal wealth} is the market incompleteness, which leads to non-unique pricing kernels and the fact that not every contingent claim can be replicated by a self-financing trading strategy. In the following, we present the proof in successive steps. First, we show that the product of any pricing kernel and any admissible wealth process, $\{\zeta_t X_t\}_{0\leq t\leq T}$, is a supermartingale. Second, we find the optimal terminal wealth for the concave envelope and show that it also solves the optimization problem of the original utility. In other words, we show that the concavification principle (see \cite{C2000}) holds in incomplete markets. Last, we show that this terminal wealth can be replicated. Here we present the proof of supermartingality to stress the importance of Assumptions \ref{assp:assumption 1}-\ref{assp:assumption 2}. The remaining part of the proof of Theorem \ref{thm:optimal wealth} is in Appendix \ref{sec:proof of optimal wealth}.

We begin with an elementary proposition. Let $\M$ denote the set of all pricing kernel processes. That is,
\begin{equation}
\M = \left\{\{\zeta_t\}_{0\leq t\leq T} : \zeta_t = \exp \left\{-\frac12 \int_0^t \| \Tilde{\pmb{\theta}}_s\|_2^2 \, \d s - \int_0^t \Tilde{\pmb{\theta}}_s^\intercal\, \d \mathbf{W}_s  \right\},\; \text{ where $\Tilde{\pmb{\theta}}_t$ is a solution of Eq. \eqref{eq:price of risk} }  \right\}.
\end{equation}
Let $\{X_t^{x_0} \}_{0\leq t\leq T}$ denote any wealth process with initial value $x_0$. Let $\X^{x_0}$ denote the set of all terminal wealth variables, i.e.,
\begin{equation}
    \X^{x_0} := \left\{ X_T^{x_0} = x_0 + \int_0^T \( r_t X_t^{x_0} + \pmb{\pi}_t^\intercal \(  \pmb{\mu}_t - r_t \mathbf{1}_m\)\) \d t + \int_0^T \pmb{\pi}_t^\intercal \pmb{\sigma}_t \d \mathbf{W}_t : \pmb{\pi} \in \mathcal{V} \right\}.
\end{equation}
\begin{proposition}
For any $\{\zeta_t \}_{0\leq t\leq T} \in \M$ and $X_T^{x_0} \in \X^{x_0} $, the process $\{\zeta_t X_t^{x_0}\}_{0\leq t\leq T}$ is a supermartingale under either Assumption \ref{assp:assumption 1} or Assumption \ref{assp:assumption 2}.
\end{proposition}

\begin{proof}
Applying Itô's formula, we get
\begin{equation}\label{eq:discounted wealth process}
\d \( \zeta_t X_t^{x_0} \) = \[\zeta_t \pmb{\pi}_t^\intercal \pmb{\sigma}_t - \zeta_t X_t^{x_0} \Tilde{\pmb{\theta}}_t^\intercal  \] \d \mathbf{W}_t,
\end{equation}
which shows that $\{\zeta_t X_t^{x_0} \}_{0\leq t\leq T}$ is a local martingale. 
If Assumption \ref{assp:assumption 2} holds, then $\{\zeta_t X_t^{x_0} \}_{0\leq t\leq T}$ is a local martingale with a lower bound and hence a supermartingale. This is a common assumption as the utility in the literature usually has a domain bounded from left. As we allow the utility functions to have a unbounded domain from left (i.e., $X_t^{x_0}$ can be arbitrarily negative and there is no lower bound), we need to propose some other assumption to make $\{\zeta_t X_t^{x_0} \}_{0\leq t\leq T}$ in Eq. \eqref{eq:discounted wealth process} indeed a martingale. Hence, Assumption \ref{assp:assumption 1} is proposed from the perspective of integrability. For the rest of this proof, we prove the statement under Assumption \ref{assp:assumption 1}.
Note that
\begin{equation}\label{eq:sufficient condition for MG}
 \quad \E \[\int_0^T \| \zeta_t \pmb{\sigma}_t^\intercal \pmb{\pi}_t - \zeta_t X_t^{x_0}\Tilde{\pmb{\theta}}_t\|_2^2 \,\ d t \]
\leq 2\cdot \( \E \[\int_0^T \zeta_t^2 \| \pmb{\sigma}_t^\intercal \pmb{\pi}_t \|_2^2 \,\ d t \] + \E \[\int_0^T \zeta_t^2 \(X_t^{x_0}\)^2 \|\Tilde{\pmb{\theta}}_t\|_2^2 \,\ d t \]\).
\end{equation}
By Hölder's inequality, for any $p, q > 1$ with $\frac1p + \frac1q = 1$, we have
\small{
\begin{equation}\label{eq:first finiteness}
\begin{aligned}
\E \[\int_0^T \zeta_t^2 \| \pmb{\sigma}_t^\intercal \pmb{\pi}_t \|_2^2 \,\ d t \] &\leq \E\[ \int_0^T \zeta_t^2 \lambda_t^{\max} \|\pmb{\pi}_t\|_2^2\, \d t \]\\
(\text{Fubini--Tonelli Theorem}) &= \int_0^T \lambda_t^{\max} \E \[ \zeta_t^2 \| \pmb{\pi}_t\|_2^2 \] \, \d t\\
(\text{Hölder's Inequality}) &\leq \( \sup_{t\in [0,T]} \lambda_t^{\max} \)\int_0^T \(\E\[ \zeta_t^{2q} \]\)^{\frac1q} \(\E\[\| \pmb{\pi}_t \|_2^{2p} \] \)^{\frac1p} \,\d t\\
(\text{Hölder's Inequality}) &\leq \( \sup_{t\in [0,T]} \lambda_t^{\max} \)\(\int_0^T \E\[ \zeta_t^{2q} \]\, \d t\)^{\frac1q} \(\int_0^T \E\[\| \pmb{\pi}_t \|_2^{2p} \]  \,\d t\)^{\frac1p}.
\end{aligned}
\end{equation}
}
Moreover, we have
\small{
\begin{align}
\E \[\int_0^T \zeta_t^{2q}\,\d t \] &= \E\[ \int_0^T \exp \left\{-q \int_0^t \|\Tilde{\pmb{\theta}}_s\|_2^2 \, \d s - 2q \int_0^t \Tilde{\pmb{\theta}}_s^\intercal \d \mathbf{W}_s \right\} \,\d t\]\nonumber\\
&\leq \E\[ \(\sup_{t \in [0,T]} \exp\left\{-q \int_0^t \| \Tilde{\pmb{\theta}}_s\|_2^2 \, \d s \right\} \) \int_0^T \exp\left\{-2q \int_0^t \Tilde{\pmb{\theta}}_s^\intercal \, \d \mathbf{W}_s \right\}\, \d t \]\nonumber\\
&\leq \E \[ 1 \cdot \int_0^T \exp\left\{ -2q \int_0^t \Tilde{\pmb{\theta}}_s^\intercal \, \d \mathbf{W}_s \right\}\, \d t \]\nonumber\\
(\text{Fubini--Tonelli Theorem})&= \int_0^T \E\[ \exp\left\{-2q \int_0^t \Tilde{\pmb{\theta}}_s^\intercal \, \d \mathbf{W}_s \right\} \]\, \d t.
\end{align}}Let $\rho_t := \int_0^t \Tilde{\pmb{\theta}}_s^\intercal\,\d \mathbf{W}_s $. Then $\rho_t \sim \mathcal{N}\(0, \int_0^t \|\Tilde{\pmb{\theta}}_s\|_2^2\,\d s \) $. Let $\Tilde{\sigma}_t := \sqrt{\int_0^t \|\Tilde{\pmb{\theta}}_s\|_2^2\,\d s } $. By direct computation, we have
{
\begin{align}
\E\[ \exp\left\{-2q\rho_t\right\} \] &= \int_\R \frac{1}{\sqrt{2\pi \Tilde{\sigma}_t^2}} \exp\left\{-2qx - \frac{x^2}{2\Tilde{\sigma}_t^2} \right\}\,\d x
= \exp\left\{2 \Tilde{\sigma}_t^2 q^2\right\}.
\end{align}
}It follows that 
{
\begin{align}
\E \[\int_0^T \zeta_t^{2q}\,\d t \] \leq \int_0^T  \E\[ \exp\left\{-2q\rho_t\right\} \]\,\d t
= \int_0^T \exp\left\{2\Tilde{\sigma}_t^2 q^2\right\}\,\d t
\leq T\cdot \exp\left\{2q^2 \int_0^T \|\Tilde{\pmb{\theta}}_s\|_2^2\,\d s\right\}
\stackrel{\text{by \eqref{eq:Novikov condition}}}{<} \infty.
\end{align}
}Now, we claim that the finiteness of the two terms on the right hand side of \eqref{eq:sufficient condition for MG} relies on that $\E \[\int_0^T \|\pmb{\pi}_t\|_2^{2p}\, \d t\]  < \infty$ for some $p > 1$; we will prove this later in details. Hence, we propose Assumption \ref{assp:assumption 1}. Under this assumption, we can find some $p = 1 + \frac{\epsilon}{2}$ and a corresponding $q$ such that every term on the right hand side of Eq. \eqref{eq:first finiteness} is finite. 
On the other hand, similar to Eq. \eqref{eq:first finiteness}, we have
\begin{equation}\label{eq:second finiteness}
\begin{aligned}
\E \[\int_0^T \zeta_t^2 \(X_t^{x_0}\)^2 \|\Tilde{\pmb{\theta}}_t\|_2^2 \,\ d t \] &\leq \(\sup_{t \in [0, T]} \|\Tilde{\pmb{\theta}}_t\|_2^2 \) \(\int_0^T \E\[ \zeta_t^{2q} \]\, \d t\)^{\frac1q} \(\int_0^T \E\[\(X_t^{x_0} \)^{2p} \]  \,\d t\)^{\frac1p}.
\end{aligned}
\end{equation}
Solving \eqref{wealthprocess} gives
\begin{equation}
X_t^{x_0} = e^{\int_0^t r_s \,\d s} \[ x_0 + \int_0^t e^{-\int_0^s r_u \,\d u} \(\pmb{\mu}_s - r_s \mathbf{1}_m \)^\intercal \pmb{\pi}_s\, \d s + \int_0^t e^{-\int_0^s r_u \, \d u} \pmb{\pi}_s^\intercal \pmb{\sigma}_s \d \mathbf{W}_s \].
\end{equation}
With this, the following inequality holds under Assumption \ref{assp:assumption 1}:
\begin{equation}
\begin{aligned}
\int_0^T \E \[ \(X_t^{x_0}\)^{2p} \]\, \d t 
&\leq T \cdot \E\[ \(\sup_{t\in[0,T]} |X_t^{x_0}| \)^{2p} \]\\
(\text{Burkholder--Davis--Gundy Inequality}) &\leq T\cdot \Tilde{K}_p \cdot \E\[ \(\langle X^{x_0} \rangle_T\)^p \]\\
&= T\cdot \Tilde{K}_p \cdot \E\[ \(\int_0^T\| \pmb{\pi}_t^\intercal \pmb{\sigma}_t\|_2^2\, \d t \)^p \]\\
&\leq T\cdot \Tilde{K}_p \cdot \E \[ \(\int_0^T \lambda_t^{\max} \|\pmb{\pi}_t\|_2^2\, \d t \)^p \]\\
&= T\cdot \Tilde{K}_p \(\sup_{t\in[0,T]}\lambda_t^{\max} \)^p \cdot \E\[\(\int_0^T\|\pmb{\pi}_t\|_2^2 \, \d t \)^p \]\\
(\text{Hölder's Inequality}) &\leq T\cdot \Tilde{K}_p \(\sup_{t\in[0,T]}\lambda_t^{\max} \)^p \cdot \E\[T^{\frac{p}{q}}\cdot \(\int_0^T \|\pmb{\pi}_t\|_2^{2p}\, \d t \)\]\\
& = T^{1 + \frac{p}{q}} \cdot \Tilde{K}_p \cdot \(\sup_{t \in [0,T]} \lambda_t^{\max}\)^p \E \[ \int_0^T \|\pmb{\pi}_t \|_2^{2p}\, \d t \]< \infty,
\end{aligned}
\end{equation}
where $\Tilde{K}_p$ is a constant depending on $p$.
Again, if $\E\[\int_0^T  \| \pmb{\pi}_t\|_2^{2+\epsilon}\, \d t \] $ is finite for some $\epsilon > 0$, we have that $\E \[\int_0^T \zeta_t^2 \(X_t^{x_0}\)^2 \|\pmb{\theta}_t\|_2^2 \d t \]$ is finite. 
By Corrollary 3.2.6 in \cite{O2003}, the two finiteness in Eqs. \eqref{eq:first finiteness} and \eqref{eq:second finiteness} guarantee that $\{\zeta_t X_t^{x_0}\}_{0\leq t\leq T}$ is an $\{\F_t\}_{0\leq t\leq T}$-martingale, and hence a supermartingale.
\end{proof}
Further, we give a proposition on the concavification technique in incomplete markets.
\begin{proposition}
Let $U^{**}$ be the concave envelope of $U$. Then the optimization problem \eqref{problem_enve} shares the same optimal solution as problem \eqref{eq:main problem}.
\end{proposition}
\begin{proof}
We first derive the concave envelope $U^{**}$ of $U$ and discuss the optimization problem \eqref{problem_enve}.
To derive the former, we utilize the convex conjugate of $U^{**} $:
\begin{equation}\label{eq:convex conjugate}
\begin{aligned}
\Tilde{U}^{**}\(\kappa \) &= \max_{x \in \text{dom} U} \left\{ U^{**}\(x\) - \kappa x \right\}\\
&= U^{**} \(I\(\kappa\) \) - \kappa I\(\kappa\),\quad \kappa > 0,
\end{aligned}
\end{equation}
where $I$ is the generalized inverse of $\(U^{**}\)' $, defined by
\begin{equation}
I(y) := \inf\left\{ x \in \text{dom}U : \(U^{**}\)'(x) \leq y \right\}.
\end{equation}
By definition of $I(\cdot)$, the function $y \mapsto \E\[ \zeta_T I\(y \zeta_T\) \]$ is strictly decreasing for any $\zeta \in \M$. Hence, for fixed $\zeta \in \M$, we can define a function $\Y_\zeta \(\cdot\)$ as the inverse of $y \mapsto \E\[ \zeta_T I\(y \zeta_T\) \]$. Substituting $\kappa =  \Y_\zeta \(x_0\)\zeta_T$ in \eqref{eq:convex conjugate}, we have
\begin{equation}\label{eq:supermartingale}
U^{**} \( I\(\Y_\zeta \(x_0\) \zeta_T \) \) - \Y_\zeta \(x_0\) \zeta_T I\( \Y_\zeta \(x_0\) \zeta_T\) \geq U^{**} \(X_T^{x_0} \) - \Y_\zeta \(x_0\) \zeta_T X_T^{x_0}.
\end{equation}
Additionally, $\E\[ \zeta_T I\( \Y_\zeta\(x_0\)\zeta_T \) \] = x_0$ and $\E\[\zeta_T X_T^{x_0} \]\leq x_0$. Taking expectation on both sides of \eqref{eq:supermartingale}, we have 
\begin{equation}
\E \[ U^{**} \( I\(\Y_\zeta \(x_0\) \zeta_T \) \) \] \geq \E \[ U^{**}\(X_T^{x_0}\) \].
\end{equation}
Since $X_T^{x_0} \in \X^{x_0}$ is arbitrary, this proves the optimality of the terminal wealth $I\(\Y_\zeta\(x_0\)\zeta_T \)$. Moreover, by the arbitrariness of $\zeta \in \M$, we have that $\E\[U^{**}\( I\(\Y_\zeta \(x_0\) \zeta_T\)\)\]$ has the same value for all $\zeta \in \M$. Hence, we use $\xi$ defined in \eqref{eq:pricing kernel} as a representative.
By definition, we know that $U^{**}(x) \geq U(x) $ for all $x \in \text{dom}U $. Hence, we have
\begin{equation}
    \E \[ U^{**}\( I\( \Y_\xi\(x_0\) \xi_T \) \) \] \geq \E \[ U^{**} \(X_T^{x_0} \) \] \geq \E \[ U\(X_T^{x_0} \) \].
\end{equation}
Define the index set
\begin{equation}
\mathcal{I} := \left\{ k \in \{0, 1, \ldots, n\} | \gamma_k^+ = \gamma_{k+1}^- \right\}.
\end{equation}
Note that $U^{**}\(x\) \neq U\(x\) $ if and only if $x \in \bigcup_{k \in \mathcal{I}} \(a_k, a_{k+1}\) $. Moreover, since $\xi_T$ is continuously distributed, we have
\begin{equation}
\p \( \Y_\xi(x_0) \xi_T \in \{\gamma_k^-, \gamma_k^+\} \) \leq \p \( \Y_\xi(x_0) \xi_T \in \{\gamma_k^-,\gamma_k^+\} \cup \{ \gamma_{n+1}^- \} \) = 0.
\end{equation}
Now for notation simplicity, let $y^* = \Y_\xi(x_0)$ and $X_T^* = I\(y^* \xi_T \)$. We have
{\small
\begin{align}
\E \[ U\(X_T^*\) \] &= \E \[ U\(X_T^*\) \mathds{1}_{\{X_T^* \in \bigcup_{k\in \mathcal{I}}\(a_k,a_{k+1} \)\} } \] + \E \[ U\(X_T^*\)\mathds{1}_{\{ X_T^* \in \R \backslash \bigcup_{k \in \mathcal{I}}\(a_k,a_{k+1}\) \}} \]\nonumber\\
&= \E \[ U\(X_T^*\) \mathds{1}_{\{ y^* \xi_T = \gamma_k^+ \} } \] + \E \[ U\(X_T^*\)\mathds{1}_{\{ X_T^* \in \R \backslash \bigcup_{k \in \mathcal{I}}\(a_k,a_{k+1}\) \}} \]\nonumber\\
&= 0 + \E \[ U\(X_T^*\)\mathds{1}_{\{ X_T^* \in \R \backslash \bigcup_{k \in \mathcal{I}}\(a_k,a_{k+1}\) \}} \]\nonumber\\
&= 0 + \E \[ U^{**}\(X_T^*\)\mathds{1}_{\{ X_T^* \in \R \backslash \bigcup_{k \in \mathcal{I}}\(a_k,a_{k+1}\) \}} \]\nonumber\\
&= \E \[ U^{**}\(X_T^*\) \mathds{1}_{\{X_T^* \in \bigcup_{k\in \mathcal{I}}\(a_k,a_{k+1} \)\} } \] + \E \[ U^{**}\(X_T^*\)\mathds{1}_{\{ X_T^* \in \R \backslash \bigcup_{k \in \mathcal{I}}\(a_k,a_{k+1}\) \}} \]\nonumber\\
&= \E \[ U^{**}\(X_T^*\) \].
\end{align}
}
This proves that the optimizer of Problem \eqref{problem_enve} also maximizes Problem \eqref{eq:main problem}.
\end{proof}

A seminal literature on expected utility maximization in incomplete markets is \cite{KLSX1991}, in which the so--called \textit{fictitious completion} thought experiment is introduced. The authors suppose that the market is completed by $(q-m)$ risky assets with the return vector $\pmb{\alpha}_t \in \R^{q-m}$ and volatility matrix $\pmb{\rho}_t \in \R^{(q-m) \times q}$, where the rows of $ \pmb{\rho}_t $ are orthonormal and in the kernel of $ \pmb{\sigma}_t $. This orthogonality, in some way, implies that the optimal portfolio in incomplete markets should coincide with the one in complete markets. We proceed to show this implication in details. 

Suppose that the market is completed by $(q-m)$ risky assets with the return vector $\pmb{\alpha}_t \in \R^{q-m}$ and volatility matrix $\pmb{\rho}_t \in \R^{(q-m) \times q}$, where the rows of $ \pmb{\rho}_t $ are orthonormal and in the kernel of $ \pmb{\sigma}_t $. Let
\begin{equation}
\pmb{\Tilde{\sigma}}_t :=
\begin{bmatrix}
\pmb{\sigma}_t\\
\pmb{\rho}_t
\end{bmatrix}
\in \R^{q\times q} ,\quad \pmb{\Tilde{\mu}}_t := 
\begin{bmatrix}
\pmb{\mu}_t\\
\pmb{\alpha}_t
\end{bmatrix}
\in \R^q,\quad t \in \[0, T\],
\end{equation}
denote respectively the augmented volatility matrix and the augmented return vector.
With a slight abuse of notations, we have the family of augmented price of risk
\begin{equation}
\pmb{\Tilde{\theta}}_{\pmb{\nu},t} := \pmb{\Tilde{\sigma}}_t^\intercal \( \pmb{\Tilde{\sigma}}_t \pmb{\Tilde{\sigma}}_t^\intercal \)^{-1} \( \pmb{\Tilde{\mu}}_t - r_t \mathbf{1}_q \)\\
= \pmb{\theta}_t + \pmb{\nu}_t,
\end{equation}
where $\pmb{\nu}_t := \pmb{\rho}_t^\intercal \( \pmb{\alpha}_t - r_t \mathbf{1}_{q-m} \) $. Observing that $\pmb{\theta}_t$ is orthogonal to $\pmb{\nu}_t$ for all $t \in [0,T]$, we can define the family of augmented pricing kernels:
\begin{equation}
\begin{aligned}
\xi_t^{\pmb{\nu}} :=& \exp\left\{ -\int_0^t \( r_s + \frac12 \| \pmb{\Tilde{\theta}}_{\pmb{\nu},s} \|_2^2 \)\, \d s - \int_0^t \pmb{\Tilde{\theta}}_{\pmb{\nu},s}^\intercal \, \d \mathbf{W}_s \right\}\\
=&\exp\left\{ -\int_0^t \( r_s + \frac12 \(\| \pmb{\theta}_t \|_2^2 + \| \pmb{\nu}_t\|_2^2\) \)\, \d s - \int_0^t \pmb{\Tilde{\theta}}_{\pmb{\nu},s}^\intercal \, \d \mathbf{W}_s \right\}.
\end{aligned}
\end{equation}
Our main goal is to show that the expected utility of our strategy is greater or equal to the strategy derived under any augmented pricing kernel. This is achieved by the following lemma.
\begin{lemma}\label{lem:optimal condition in KLSX1991}
Let $X_T^{\pmb{\nu}}$ denote the optimal terminal wealth under the pricing kernel $\{\xi_t^{\pmb{\nu}}\}_{0\leq t \leq T}$, $X_T^*$ denote the optimal wealth under the original pricing kernel $\{\xi_t\}_{0\leq t \leq T}$, and $x_0 > 0$ denote the initial value of the wealth process. If
\begin{equation}\label{eq:financiability}
\mathbb{E} \[ \xi_T^{\pmb{\nu}}X_T^* \] \leq x_0, \quad \text{for all}\; \; \pmb{\nu}_t \in \ker \( \pmb{\sigma}_t \) , 
\end{equation}
where
$
\ker \( \pmb{\sigma}_t \) := \{ \mathbf{v} \in \R^q : \pmb{\sigma}_t \mathbf{v} = \mathbf{0} \}
$, then
the optimal portfolio $\{\pmb{\pi}_t^*\}_{0\leq t\leq T}$ derived under $\{\xi_t\}_{0\leq t \leq T}$ is optimal in the incomplete market.
\end{lemma}
The proof is referred to Theorem 9.3 in \cite{KLSX1991}. In the original context, the domain of the utility function $U$ is assumed to be $\R_+$. However, we note the the proof of Theorem 9.3 is not based on this assumption. Hence, the theorem can be directly applied to the PSAHARA utilities.

We explore further to find the sufficient condition that Eq. \eqref{eq:financiability} holds.
\begin{proposition}\label{prop:developed optimal condition}
Under the assumption given by Eq. \eqref{eq:Novikov condition}, then $\{\pmb{\pi}_t^*\}_{0\leq t\leq T}$ in Theorem \ref{optimal portfolio} is the optimal portfolio in the incomplete market, i.e.,
\begin{equation}
\mathbb{E} \[ U\( X_T^* \) \] \geq \mathbb{E} \[ U\( X_T^{\pmb{\nu}} \) \],
\end{equation}
for any $ \{ \pmb{\nu}_t \}_{0\leq t \leq T} \in \ker \(\pmb{\sigma}_t\) $.
\end{proposition}

\begin{proof}
We only need to show that if $\{\pmb{\nu}_t\}_{0\leq t \leq T}$ satisfies Eq. \eqref{eq:Novikov condition}, then $\mathbb{E}\[ \xi_T^{\pmb{\nu}} X_T^*\] \leq x_0$.
Define the new probability measure $\hat{\p}$, with
\begin{equation}\label{eq:RN derivative}
\frac{\d \hat{\p} }{\d \p} = \exp \left\{ -\int_0^T \pmb{\nu}_s^\intercal \, \d \mathbf{W}_s - \frac12 \int_0^T \| \pmb{\nu}_s \|_2^2 \, \d s \right\} = \frac{\xi_T^{\pmb{\nu}}}{\xi_T}.
\end{equation}
We see that if $\{\pmb{\nu}_t\}_{0\leq t\leq T}$ satisfies Eq. \eqref{eq:Novikov condition}, $\hat{\p}$ is a well-defined probability measure and $\frac{\d \hat{\p}}{\d \p}$ is the Radon--Nikodym derivative. By Girsanov's Theorem, the components of Brownian motion $\{\hat{\mathbf{W}}_t\}_{0\leq t\leq T} $ under $\hat{\p}$ is 
\begin{equation}
\hat{W}_{i,t} = W_{i,t} + \int_0^t \nu_{i,s} \, \d s,\quad i = 1, \ldots, q.
\end{equation}
Moreover, since $\pmb{\theta}_t$ is orthogonal to $\pmb{\nu}_t$ for all $t \in [0,T]$, we have
{\small
\begin{align}
\xi_T &= \exp \left\{ -\int_0^T \( r_s + \frac12 \| \pmb{\theta}_s \|_2^2 \) \, \d s - \int_0^T \pmb{\theta}_s^\intercal \, \d \mathbf{W}_s \right\}\nonumber\\
&= \exp \left\{ -\int_0^T \( r_s + \frac12 \| \pmb{\theta}_s \|_2^2 \) \, \d s +\int_t^T \pmb{\theta}_s^\intercal \pmb{\nu}_s \d s - \int_0^T \pmb{\theta}_s^\intercal \, \d \mathbf{\hat{W}}_s \right\}\nonumber\\
&= \exp \left\{ -\int_0^T \( r_s + \frac12 \| \pmb{\theta}_s \|_2^2 \) \, \d s - \int_0^T \pmb{\theta}_s^\intercal \, \d \mathbf{\hat{W}}_s \right\}.
\end{align}
}This shows that $\xi_T$ has the same distribution under $\p$ and $\hat{\p}$. Note that $X_T^* = I \(y^*\xi_T\)$ is a function of $\xi_T$. Define $h(z) := z I\(y^* z\)$. Clearly, we have $\E\[h(\xi_T)\] = \hat{\E}\[h(\xi_T) \] $, where $\hat{\E}$ is the expectation under $\hat{\p}$. 
Combining with \eqref{eq:RN derivative}, it directly follows that
\begin{equation}
\mathbb{E} \[ \xi_T^{\pmb{\nu}}X_T^* \] =\mathbb{\hat{E}} \[ \xi_T X_T^* \]= \mathbb{E} \[\xi_T X_T^*\] = x_0.
\end{equation}
\end{proof}
Lemma \ref{lem:optimal condition in KLSX1991} and Proposition \ref{prop:developed optimal condition} show that if Eq. \eqref{eq:Novikov condition} holds, the optimality of $\{\pmb{\pi}_t^*\}_{0\leq t\leq T}$ in Theorem \ref{optimal portfolio} holds in incomplete markets.

\section{Analysis on the Optimal Processes}\label{sec:analysis}
In this section, we conduct an asymptotic analysis and a numerical analysis on the optimal wealth process and the optimal portfolio given in Section \ref{sec:optimal portfolio} to show some characteristics of these processes.
\subsection{Asymptotic Analysis}\label{asymptotic analysis}
We conduct the asymptotic analysis for PSAHARA utilities to illustrate the economic insight in Theorem \ref{optimal portfolio}; see also \cite{LL2024}. Since $\xi_t$ indicates the market state (the smaller $\xi_t$, the better the market; see \cite{C2000}), we can study the risk-taking behaviors of the portfolio by asymptotic analysis on $\xi_t$. One can apply Theorem \ref{thm:asymptotic analysis} directly to any PSAHARA utility. The proof is stated in Appendix \ref{proof of aymptotic analysis}.
\begin{theorem}\label{thm:asymptotic analysis}
Suppose the settings in Theorem \ref{optimal portfolio} hold. For fixed $t \in [0, T)$ and $y^* \in (0, \infty)$,
\begin{enumerate}
    \item [(a)] as $\xi_t \rightarrow 0$, we have
    \begin{equation}
    X_t^* \rightarrow \infty, \quad \frac{\pmb{\pi}_t^*}{X_t^*} \rightarrow \frac{1}{\alpha_n} \pmb{\sigma}_t^\intercal \( \pmb{\sigma}_t \pmb{\sigma}_t^\intercal \)^{-1} \( \pmb{\mu}_t - r_t \mathbf{1}_m \);
    \end{equation}
    \item [(b)] as $\xi_t \rightarrow \infty$, we have
    \begin{equation}
    X_t^* \rightarrow -\infty, \quad \frac{\pmb{\pi}_t^*}{X_t^*} \rightarrow -\frac{1}{\alpha_0} \pmb{\sigma}_t^\intercal \( \pmb{\sigma}_t \pmb{\sigma}_t^\intercal \)^{-1} \( \pmb{\mu}_t - r_t \mathbf{1}_m \).
    \end{equation}
\end{enumerate}
The details of the asymptotic analysis of each term are given in Tables \ref{asymptotic table X}--\ref{asymptotic table pi}.
\end{theorem}

\begin{table}[h!]
\centering
\begin{tabular}{|c|c|c|c|c|c|}
\hline
\; & $X_t^*$ & $X_t^R$ & $X_t^{\bar{R}}$ & $X_t^B$ & $X_t^D$\\
\hline
$\xi_t \rightarrow 0$ & $\infty $&$\infty$ & 0 & $d_n e^{-\int_t^T r_s \, \d s} $ & 0\\
\hline
$\xi_t \rightarrow \infty$ & $-\infty$ & 0 & $-\infty$ & $d_0 e^{-\int_t^T r_s \, \d s} $ & 0\\
\hline
\end{tabular}
\caption{Asymptotic analysis for $X_t^*$.}\label{asymptotic table X}
\end{table}

\begin{table}[h!]
\centering
\begin{tabular}{|c|c|c|c|c|c|}
\hline
\; & $\pmb{\pi}_t^*$ & $\frac{\pmb{\pi}_t^*}{X_t^*}$& $\pmb{\pi}_t^{(2)}$ & $\pmb{\pi}_t^{(3)}$ & $\pmb{\pi}_t^{(4)} $  \\
\hline
$\xi_t \rightarrow 0$ & $\pmb{\infty}$ & $\frac{1}{\alpha_n} \pmb{\sigma}_t^\intercal \( \pmb{\sigma}_t \pmb{\sigma}_t^\intercal \)^{-1} \( \pmb{\mu}_t - r_t \mathbf{1}_m \)$& $\mathbf{0}$ & $\mathbf{0}$ & $\mathbf{0} $ \\
\hline
$\xi_t \rightarrow \infty$ & $\pmb{\infty}$ & $-\frac{1}{\alpha_0} \pmb{\sigma}_t^\intercal \( \pmb{\sigma}_t \pmb{\sigma}_t^\intercal \)^{-1} \( \pmb{\mu}_t - r_t \mathbf{1}_m \) $& $\mathbf{0}$ & $\mathbf{0}$& $\mathbf{0}$ \\
\hline
\end{tabular}
\caption{Asymptotic analysis for $\pi_t^*$}\label{asymptotic table pi}
\end{table}
As shown in Table \ref{asymptotic table X}, the behaviors of $X_t^*$ provide a clear insight into the impact of market states on the risk preferences of investors. In the scenario where $\xi_t \rightarrow 0$, indicating an extremely favorable market state, the positive risk-seeking term $X_t^R$ drives the optimal wealth $X_t^*$ to approach infinity and the negative risk-seeking term $X_t^{\bar{R}}$ tends to zero. This reflects an investor's inclination towards risk exposure supported by the optimistic prediction of the market. The loss aversion term $X_t^B$ is given by $d_n e^{-r(T-t)}$, signifying that even in favorable states, there is a potential of loss aversion regarding the threshold level. The term $X_t^D$ remains zero, highlighting that in highly favorable states, the impact of non-differentiable points vanishes.

Conversely, when $\xi_t \rightarrow \infty$, symbolizing an extremely unfavorable market state, $X_t^{\bar{R}}$ drives $X_t^*$ to negative infinity while $X_t^R$ remains at zero. Under such a scenario, the loss aversion term $X_t^B$ tends to $d_0 e^{-r(T-t)}$. Similarly as above, the non-differentiable points do not affect the optimal wealth in such an extreme scenario, whereas the first threshold level $d_0$ results in a positive $X_t^B$.

Shifting focus to the dynamics of $\pmb{\pi}_t^*$ in Table \ref{asymptotic table pi}, we observe that the optimal investment strategy $\pmb{\pi}_t^*$ tends towards infinity in both extreme scenarios. The ratio $\frac{\pmb{\pi}_t^*}{X_t^*}$ transitions from $\frac{1}{\alpha_n}\pmb{\sigma}_t^\intercal \( \pmb{\sigma}_t \pmb{\sigma}_t^\intercal \)^{-1} \( \pmb{\mu}_t - r_t \mathbf{1}_m \)$ in a favorable state to $-\frac{1}{\alpha_0}\pmb{\sigma}_t^\intercal \( \pmb{\sigma}_t \pmb{\sigma}_t^\intercal \)^{-1} \( \pmb{\mu}_t - r_t \mathbf{1}_m \)$ in an unfavorable state, illustrating risk-seeking behaviors in both favorable and unfavorable market states. Also, both terms coincide with the structure of Merton strategy of CRRA utilities as shown in \cite{M1969}. 
The terms $\pmb{\pi}_t^{(2)}$, $\pmb{\pi}_t^{(3)}$, and $\pmb{\pi}_t^{(4)}$ remain $\mathbf{0}$, indicating that these factors have little influence on the optimal portfolio under the asymptotic scenarios.
\subsection{Numerical Analysis}\label{numerical}
\begin{figure}[h!]
\centering
\includegraphics[width=0.49\linewidth]{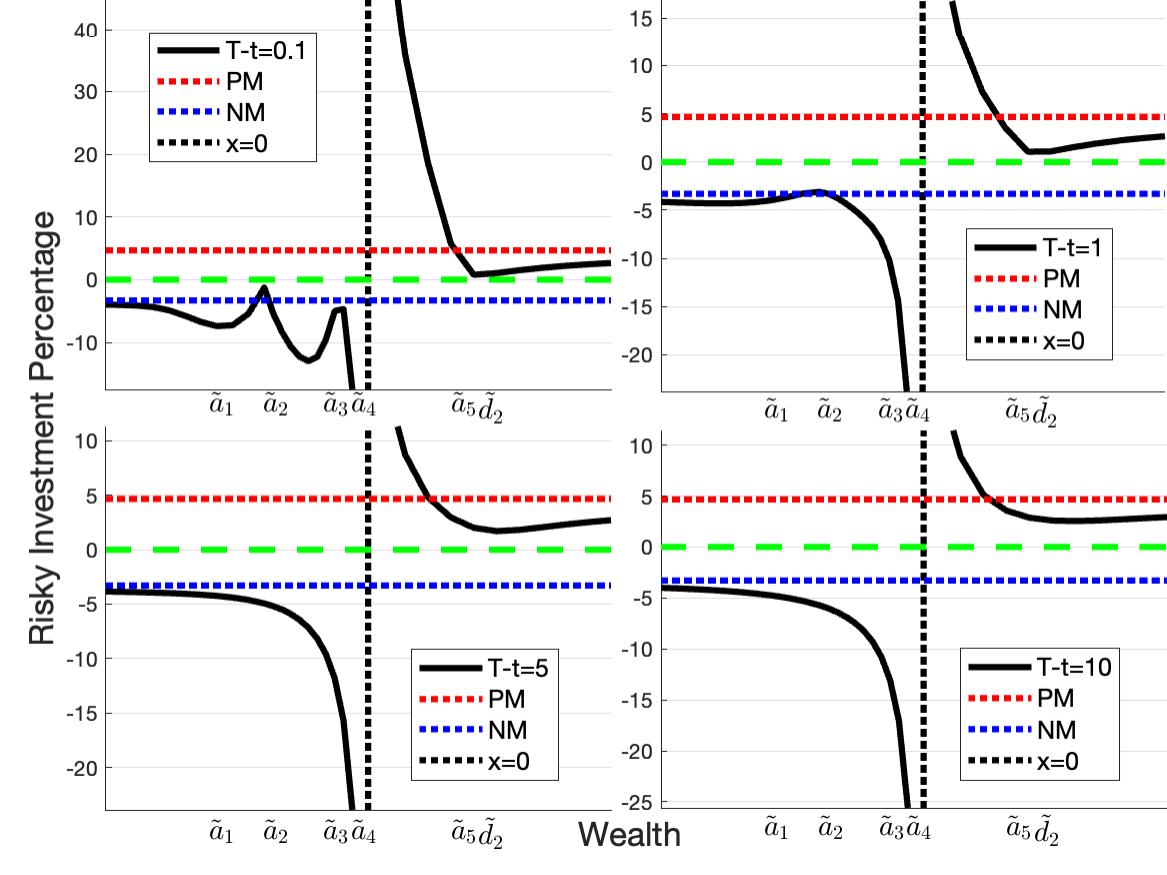}
\includegraphics[width=0.49\linewidth]{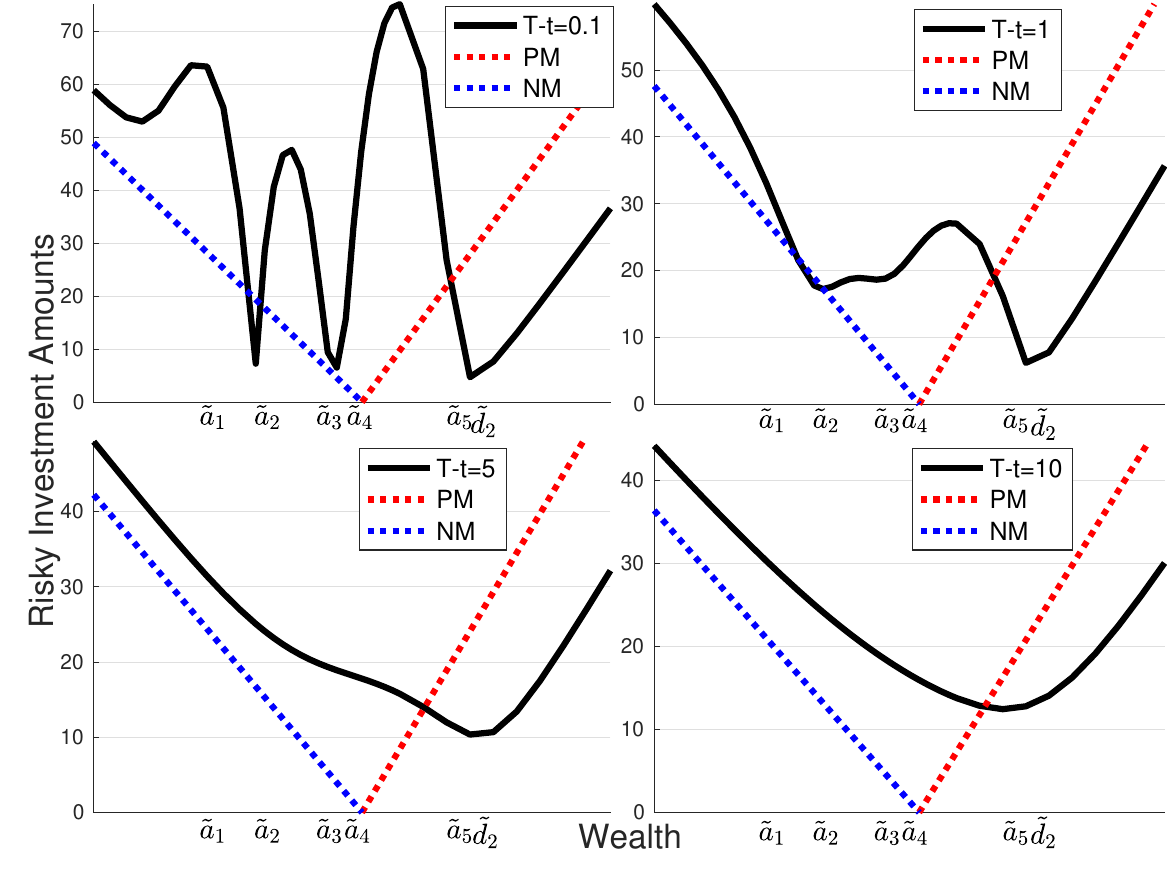}
\caption{Dynamics of $\frac{\pi_t^*}{X_t^*}$ and $\pi_t^*$ of the utility in Example \ref{eg:main example} with respect to the optimal wealth process $X_t^*$. $\Tilde{a}_i := a_i e^{-r(T-t)}, i= 1,\ldots,5, \Tilde{d}_2 := d_2 e^{-r(T-t)}. $ The positive Merton (PM in figures) line is $\pi_t^*/X_t^* = \frac{\theta}{\alpha_n\sigma}$ and the negative Merton (NM in figures) line is $\pi_t^*/X_t^* = - \frac{\theta}{\alpha_0 \sigma} $. The risky investment percentage tends to $\infty$ on the right hand side of the point $x = 0$ and to $-\infty$ on the left since the optimal portfolio remains positive for all $x \in \R$. $d_0, d_1$ lie on the linear segment $(a_4,a_5)$ and are hence omitted.}\label{fig:dynamics}
\end{figure}
In the section, we assume that the market consists of a risk-free asset and a risky asset driven by a one-dimensional Brownian motion with constant market parameters $\mu = 0.086, \sigma = 0.1, r = 0.03$. We delve into the dynamics of both the optimal risky investment percentage $\pi_t^*/X_t^*$ and the optimal portfolio $\pi_t^*$ of the utility function in Example \ref{eg:main example}. 

We mainly analyze the left graph at $T-t = 0.1$ to study the influence of linear segments $[a_1,a_2], [a_2,a_3], [a_4, a_5]$ and non-differentiable points $a_2$ and $a_4$ on the optimal portfolio. It is evident from the analysis that non-differentiable points invariably cause a decrease in the optimal risky investment percentage. Conversely, linear sections all lead to a significant increase in the optimal portfolio. Furthermore, at all tangent points ($a_3$ and $a_5$) where the graph transitions from linearity to strictly concave segments, we always observe a pattern of decline from previously elevated levels of the risky investment.

Comparing the left figures in Figure \ref{fig:dynamics} in terms of time, we see that the risky investment percentage tends to a stable ``hyperbola'' shape on the whole real line and an ``uptick'' shape on the positive wealth. To the contrast, works of \cite{C2000}, \cite{HJ2007} and \cite{LLMV2024} show that the incentive compensation induces a ``peak-valley'' pattern when the preference is PHARA. The common feature of PSAHARA and PHARA is that the non-concave and non-differentiable aspects are caused by an incentive contract, and as the contract approaches its maturity, the manager pursues more eagerly the high-profit incentive by increasing the risky investment, as shown in the subfigures of $T-t=0.1$.

To explore the effects of threshold levels on the optimal portfolio, we observe that the threshold levels $d_0$ and $d_1$ fall within the linear sections of the concave envelope, whereas $d_2$ lies exclusively on the strictly concave segment over the interval $(a_5,\infty)$. Intriguingly, the optimal portfolio exhibits symmetric behaviors around $d_2$ when $T-t$ is large.
The manager's risky investment increases as the fund value deviates from the threshold level in either directions.

Finally, we observe that the optimal investment amount $\pi_t^*$ is always strictly positive on $\R$, implying that the PSAHARA portfolio is generally very risk-seeking. This results in the extreme risk-taking behaviors of optimal investment percentage around $0$, which can be regarded as a tremendous leverage on risky investment in practice. We see from the left panel that the risky investment percentage tends to $\infty$ as $X_t^* \rightarrow 0^+$ and $-\infty$ as $X_t^* \rightarrow 0^-$. Hence, we conclude that a manager with PSAHARA utility become extremely risk-seeking when the wealth levels approach $0$.

\section{Application: Motivating Example Revisiting}\label{sec:application}
In this section, we apply our findings to the motivating example in Section \ref{sec:motivation}. Recall that this example reflects non-monotone risk aversion (SAHARA preference) and convex compensation (incentive contracts) in hedge funds.
\subsection{Optimal Portfolio}\label{sec:app opt portfolio}
\begin{figure}[h!]
    \begin{minipage}{0.48\textwidth}
    \centering
    \includegraphics[width=0.8\textwidth]{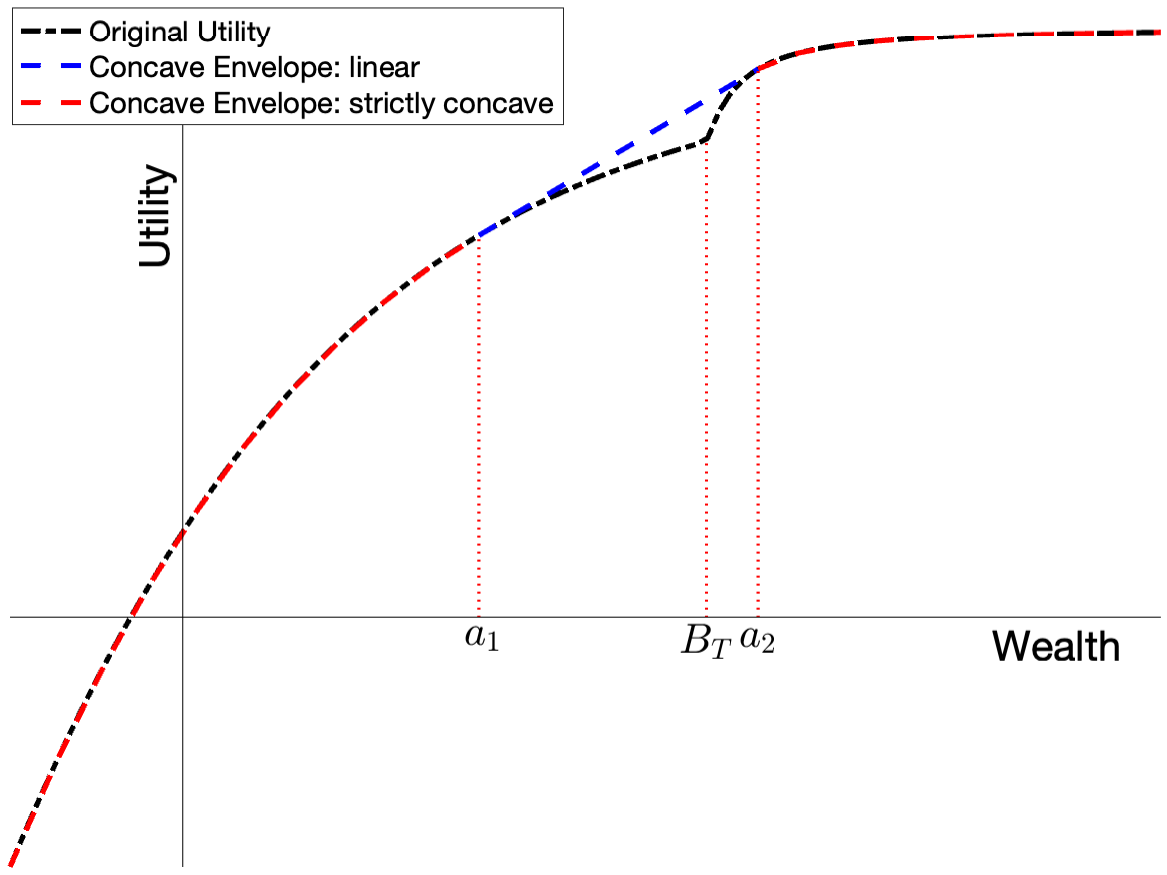}
    \end{minipage}
    \hspace{1em plus 1fill}
    \begin{minipage}{0.48\textwidth}
    \begin{equation*}
    U^{**}\(x\) := \left\{
    \begin{aligned}
    & \frac{1}{0.02} \hat{U}\(x; 2, \frac{1}{0.02^2}, 0\) + u_0, && x < a_1;\\
    & \frac{1}{0.02} \hat{U}'\(a_1; 2, \frac{1}{0.02^2},0 \) \( x - a_1\)\\
    &\quad + \hat{U}\(a_1; 2, \frac{1}{0.02^2}, 0\), && a_1 \leq x < a_2;\\
    & \frac{1}{0.22} \hat{U}\(x; 2, \frac{1}{0.22^2}, \frac{0.2B_T}{0.22}\) + u_2, && x \geq a_2.\\
    \end{aligned}
    \right.
    \end{equation*}
    \end{minipage}
    \caption{The original utility \eqref{eq:incentive utility} and the concave envelope. In the figure, the red parts are the segments where the concave envelope coincides with the original utility, while the blue straight lines are the segments where the concave envelope does not coincide with the original utility. $\hat{U}$ is the utility function shown in \eqref{basic_U} and $\hat{U}'$ is its derivative shown in \eqref{eq:derivative}. $a_1, a_2$ are the tangent points between linear segments and strictly concave parts. $u_1, u_2$ denote the correction constants to make the utility function continuous.}
    \label{fig:incentive envelope}
\end{figure}
By treating the composed utility \eqref{eq:incentive utility} as a special case of the PSAHARA utility, we can apply Theorems \ref{thm:optimal wealth} and \ref{optimal portfolio} to find the corresponding optimal wealth process and the optimal portfolio.
To better illustrate the example, we set $w = 0.2, v=0.02, X_0 = 1, B_T = X_0 e^{\int_0^T r_s \, \d s}$ in \eqref{eq:incentive contract}, where $X_0$ represents the initial fund value. We let the utility parameters be $\alpha = 2, \beta = 1, d = 0$.
In this case, the concave envelope of the composed utility is shown in Figure \ref{fig:incentive envelope}.

Moreover, we state the following theorem of the optimal portfolio of the composed utility \eqref{eq:incentive utility}.
\begin{theorem}\label{thm:incentive portfolio}
Define the portfolio selection problem in hedge funds as 
\begin{equation}\label{eq:incentive problem}
    \max_{\pmb{\pi} \in \mathcal{V}} \E \[ U\(X_T\) \],
\end{equation}
where $U := \hat{U} \circ \Theta$ is given by Eqs. \eqref{eq:simple utility} and \eqref{eq:incentive utility}.
The optimal investment strategy $\{\pmb{\pi}_t^*\}_{0 \leq t \leq T}$ for Problem \eqref{eq:incentive problem} is given by
\begin{equation}\label{eq:incentive strategy}
\pmb{\pi}_t^* =\hat{\pmb{\pi}}_t^{(1)} + \hat{\pmb{\pi}}_t^{(2)} + \hat{\pmb{\pi}}_t^{(3)},
\end{equation}
where
{\small
\begin{align}\label{eq:incentive portfolio}
\hat{\pmb{\pi}}_t^{(1)} =& \( \pmb{\sigma}_t\pmb{\sigma}_t^\intercal \)^{-1} \( \pmb{\mu}_t - r_t \mathbf{1}_m \)\frac{1}{\alpha}\( \sqrt{ \( X_{t,1}^R + X_{t,1}^{\bar{R}} \)^2 + b_{t,1}} + \sqrt{ \(X_{t,2}^R + X_{t,2}^{\bar{R}} \)^2 + b_{t,2} } \),\nonumber\\
\hat{\pmb{\pi}}_t^{(2)} =& -\frac{\( \pmb{\sigma}_t\pmb{\sigma}_t^\intercal \)^{-1} \( \pmb{\mu}_t - r_t \mathbf{1}_m \)}{2\sqrt{\int_t^T \| \pmb{\theta}_s \|_2^2\,\d s}}\left\{ e^{\(-1 + \frac{1}{\alpha} \)\int_t^T \(r_s + \frac{1}{\alpha}\|\pmb{\theta}_s\|_2^2 \)\,\d s } \[\( \frac{v^{1-\alpha} }{y^* \xi_t} \)^{\frac{1}{\alpha}}\Phi'\( g_{1}\( \frac{ U'\(a_1^-\) }{y^* \xi_t}\)\)\right.\right.\nonumber\\
&\left.-\( \frac{\(w+v\)^{1-\alpha} }{y^* \xi_t} \)^{\frac{1}{\alpha}} \Phi'\( g_{1}\( \frac{ U'\(a_2^+\)}{y^*\xi_t} \) \) \] - e^{\(-1 - \frac{1}{\alpha} \)\int_t^T \(r_s - \frac{1}{\alpha}\|\pmb{\theta}_s\|_2^2 \)\,\d s } \[ \( \frac{v^{1-\alpha} }{y^* \xi_t} \)^{-\frac{1}{\alpha}}\right.\nonumber\\
&\left.\left. \times \frac{\beta^2}{v^2} \Phi'\( g_{2}\( \frac{ U'\(a_1^-\) }{y^* \xi_t}\)\)- \( \frac{\(w+v\)^{1-\alpha} }{y^* \xi_t} \)^{-\frac{1}{\alpha}}\frac{\beta^2}{\(w+v\)^2} \Phi'\( g_{2}\( \frac{ U'\(a_2^+\)}{y^*\xi_t} \) \) \] \right\}, \nonumber\\
\hat{\pmb{\pi}}_t^{(3)} =& -\frac{\( \pmb{\sigma}_t\pmb{\sigma}_t^\intercal \)^{-1} \( \pmb{\mu}_t - r_t \mathbf{1}_m \)}{\sqrt{\int_t^T \| \pmb{\theta}_s \|_2^2\,\d s}} e^{-\int_t^T r_s \, \d s}\[\frac{d}{v}\Phi'\( g_{0}\( \frac{ U'\(a_1^-\) }{y^* \xi_t}\)\) -\frac{wB_T+d}{w+v}\Phi'\( g_{0}\( \frac{ U'\(a_2^+\)}{y^*\xi_t} \) \) \],
\end{align}}with $y^*$ given in Eq. \eqref{eq:Lagrange multiplier} and
\begin{equation}
g_1\(z\) := g_0\(z\) - \frac{1}{\alpha}\sqrt{\int_t^T \| \pmb{\theta}_s\|_2^2\, \d s}, \quad g_2\(z\) := g_0\(z\) + \frac{1}{\alpha}\sqrt{\int_t^T \| \pmb{\theta}_s\|_2^2\, \d s}.
\end{equation}
\end{theorem}
Note that the concave envelope shown in Figure \ref{fig:incentive envelope} is differentiable over its entire domain. Hence, there is no first-order risk aversion term in the optimal portfolio \eqref{eq:incentive strategy}. 
\subsection{Numerical Analysis}
Similarly as in Section \ref{numerical}, we can plot the risky investment amounts $\pmb{\pi}_t^*$ and the risky investment percentages $\frac{\pmb{\pi}_t^*}{X_t^*}$. 
\begin{figure}[h!]
\centering
\includegraphics[width=0.49\linewidth]{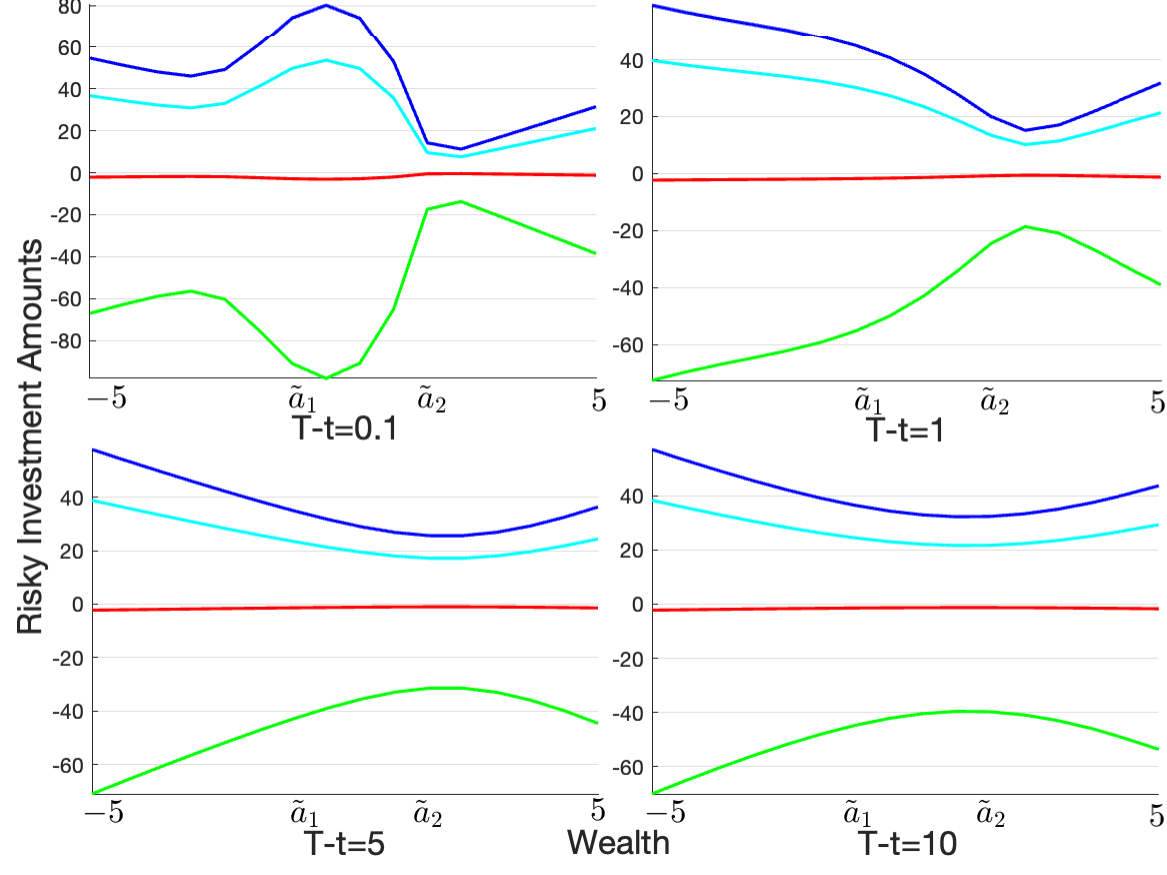}
\includegraphics[width=0.49\linewidth]{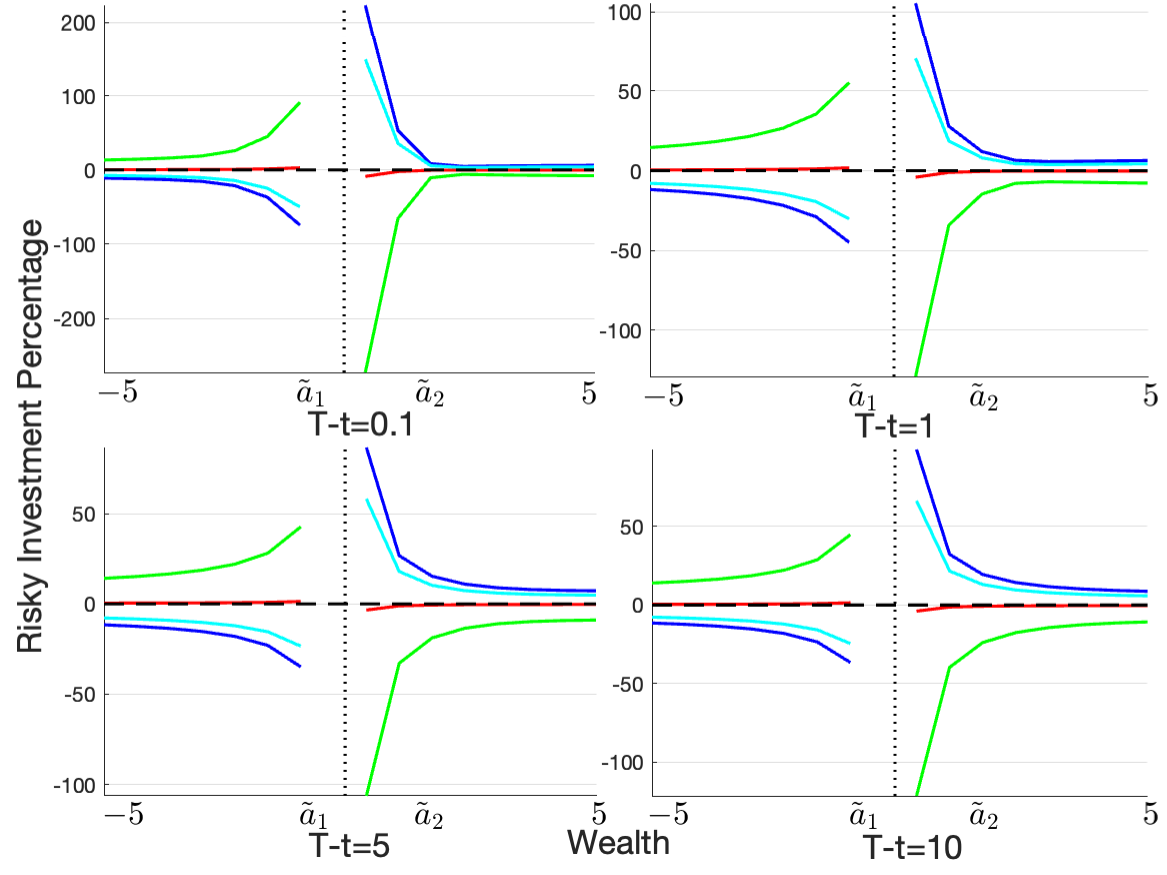}
\caption{Optimal risky investment amounts and percentages at different time, where each color represents one component of the vectors $\Tilde{a}_i :=a_i e^{-\int_t^T r_s\,\d s}, i=1,2$. In the right panel, the graphs around $X_t^* = 0$ are omitted since every component of $\frac{\pmb{\pi}_t^*}{X_t^*}$ tends to $\infty$ or $-\infty$ as $X_t^* \rightarrow 0$.}
\label{fig:incentive amounts}
\end{figure}
We assume that the fund manager operates in a four-dimensional complete market, i.$\;$e., $m = q = 4$. We present the optimal investment amounts $\pmb{\pi}_t^*$ and percentages $\frac{\pmb{\pi}_t^*}{X_t^*}$ for different wealth levels in Figure \ref{fig:incentive amounts}. Note that $[a_1,a_2]$ is the linear segment of the concave envelope. By calculation from the left panel of Figure \ref{fig:incentive amounts}, the total risky investment, i.e., the sum of optimal portfolio $\sum_{i=1}^{4} \pi_{i,t}^*$, remains strictly positive for all $X_t^* \in \R$. Furthermore, similar to the findings in \cite{C2000}, there is an increase in risky investment amounts at the linear part $[a_1,a_2]$ of the domain of $U^{**}$ in the right hand side of Figure \ref{fig:incentive amounts}, followed by a drop at the second tangent point $a_2$. This phenomenon is induced by the incentive scheme. As shown in Figure \ref{fig:incentive envelope}, $B_T$ lies in the interval $[a_1,a_2]$. In fact, $a_2$ plays a role of the benchmark level in practice. When the fund value is close but below the benchmark level, the fund manager invests heavily in risky assets to activate the call option. Upon achieving the benchmark, he employs a ``lock-in'' strategy, i.e., predominantly investing the fund in the risk-free asset to prevent falling below the benchmark level. Hence, the manager is risk-seeking on $[a_1,a_2]$ and is risk-averse on $(a_2, \infty)$.
\subsection{Empirical Study}\label{sec:empirical}
Next, we implement our optimal portfolio strategy on the U.S. stock market. The largest challenge in doing this is to estimate the value of $\{\pmb{\sigma}_t \}_{0 \leq t \leq T}$ due to the volatility ``smile" and ``smirk" phenomena in the market. We include the detailed procedures of estimation in Appendix \ref{appendix:empirical} and directly show the results in this section. We conclude that no matter what estimation method we use for volatility, due to the extreme risk-taking behaviors induced by the PSAHARA portfolio, the empirical results show great volatility, which is emphasized in the Sharpe ratio analysis.

Our analysis spans a ten-year period from March 26, 2014, to March 26, 2024. We choose major market indices—S\&P 500, NASDAQ Composite, Russell 2000, and Dow Jones Industrial Average—as proxies to implement and test the strategy shown in Theorem \ref{thm:incentive portfolio} with the parameters settings in Section \ref{sec:app opt portfolio}. We run 10000 simulations of our trading strategy over the subsequent period, from March 26, 2022, to March 26, 2024. We begin with the analysis of simple return $R_s$, which comes directly from
\begin{equation}
R_s = \frac{X_T-X_0}{X_0},
\end{equation}
with $X_0$ and $X_T$ representing the initial wealth and the terminal wealth, respectively. In the following, we set $X_0 = 1$.
\begin{figure}[t]
\centering
\includegraphics[width=0.45\linewidth]{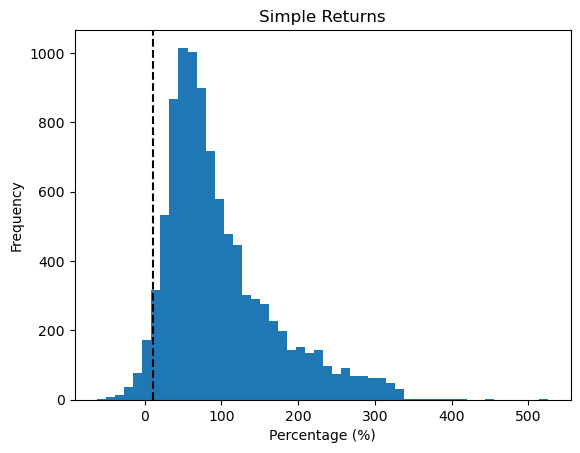}
\includegraphics[width=0.45\linewidth]{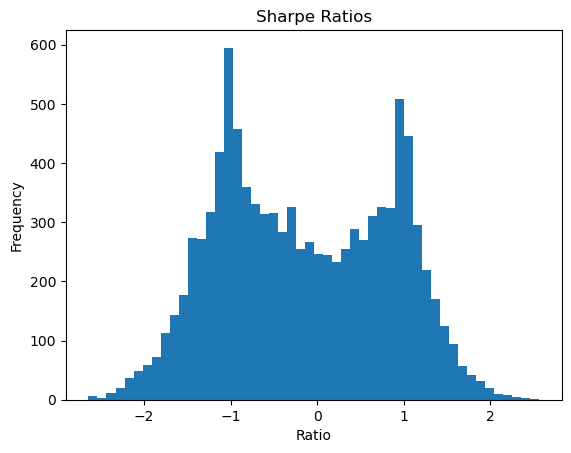}
\caption{Simple returns and Sharpe ratios of the simulations with historical volatility. The dash line means the risk-free rate $R_f$.}
\label{fig:returns and sr}
\end{figure}

As shown in Figure \ref{fig:returns and sr}, we see that the most of simple returns are positive, indicating that the strategy makes profit on average. However, we see a great deviation from the sample mean in the figures. Hence, we introduce the Sharpe ratio to evaluate the risk-adjusted returns of the strategy; see \cite{S1994}. Letting $R_f$ denote the risk-free rate and $\sigma_r$ denote the standard deviation of the difference between daily returns and the risk-free rate, we have that the Sharpe ratio $S_r$ is given by
\begin{equation}\label{eq:sharpe ratio}
S_r = \frac{R_r-R_f}{\sigma_r},
\end{equation}
where $R_r$ is the weighted average of all daily returns. 
\begin{figure}[t]
\centering
\includegraphics[width=0.45\linewidth]{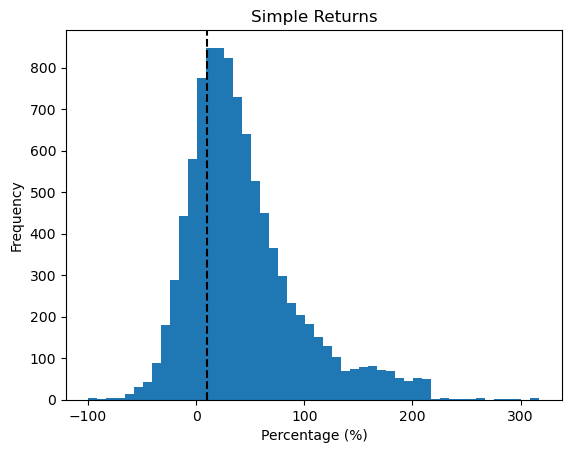}
\includegraphics[width=0.45\linewidth]{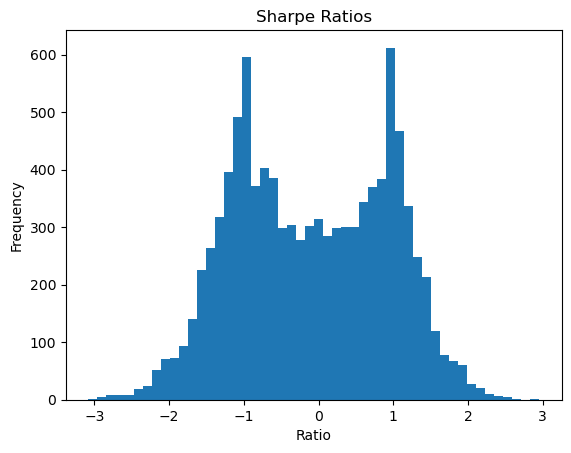}
\caption{Simple returns and Sharpe ratios of the simulations with implied volatility. The dash line means the risk-free rate $R_f$.}
\label{fig:implied volatility}
\end{figure}
\begin{figure}[t]
\centering
\includegraphics[width=0.45\linewidth]{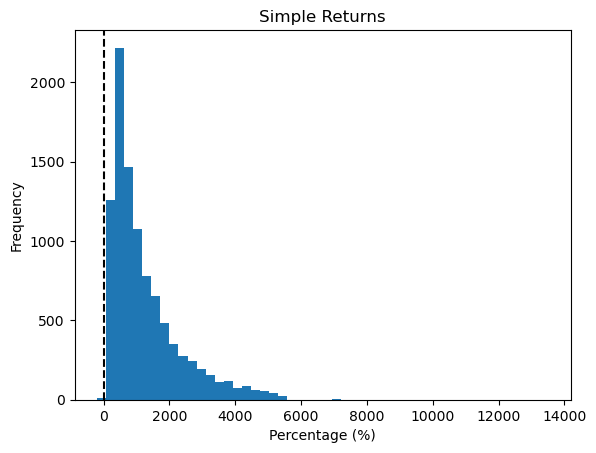}
\includegraphics[width=0.45\linewidth]{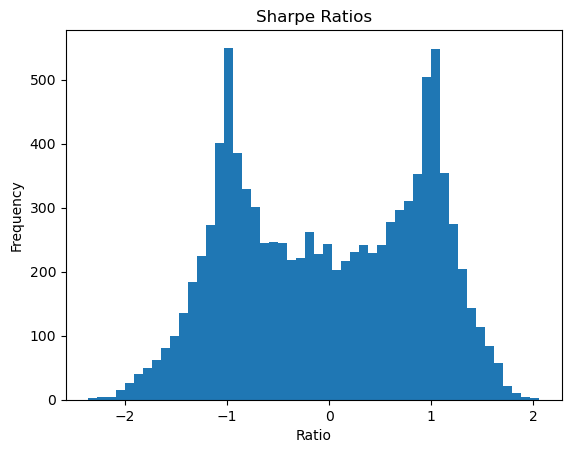}
\caption{Simple returns and Sharpe ratios of the simulations with the maximum likelihood estimator. The dash line means the risk-free rate $R_f$.}
\label{fig:MLE}
\end{figure}
Further, we use the estimation methods of implied volatility in Figure \ref{fig:implied volatility} and of the maximum likelihood estimator (MLE) in Figure \ref{fig:MLE}. The details of estimation are included in Appendix \ref{appendix:empirical}. As shown in Figures \ref{fig:returns and sr}--\ref{fig:MLE}, there is always a two-peak pattern in the Sharpe ratios. This means that the portfolio sometimes makes a high profit and sometimes makes a big loss. Moreover, we see that there exist more than half of the Sharpe ratios lying below 0. The reason is that the PSAHARA portfolio leads to a highly risk-taking behavior and causes a very negative daily return in many scenarios. We show an example of the wealth process in Figure \ref{fig:wealth process}.
\begin{figure}[h!]
    \centering
    \includegraphics[width=0.45\linewidth]{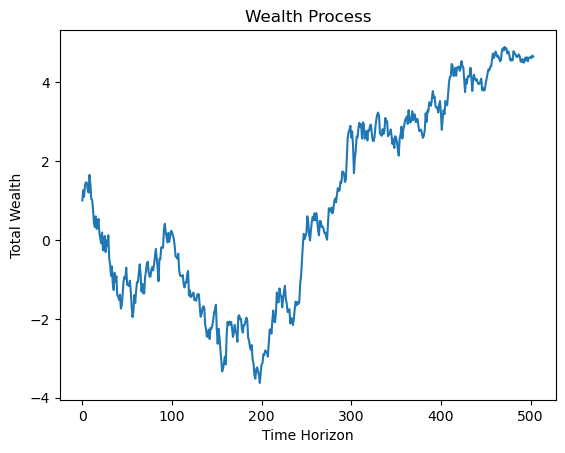}
    \caption{A sample path of wealth process $\{X_t^*\}_{0\leq t\leq T}$. The x-axis represents the trade days.}
    \label{fig:wealth process}
\end{figure}
Though achieving high returns in the end, the fund value drops to almost $-4$ times of its initial value. Particularly, some drops are of so great percentage that they draw the average to negative.

A theoretical interpretation of the Sharpe ratios' two-peak pattern is the ``gambling" behavior induced by the compensation scheme. Practically speaking, the SAHARA manager under the incentive scheme optimizes his objective by investing excessive amounts in the risky assets when the current wealth is below the incentive benchmark. Such behaviors incur a great volatility, resulting in the two-peak pattern shown in the figures. The volatility is so large that though achieving overall positive returns, the Sharpe ratios stay negative. Given this evidence, we conclude that the PSAHARA utility can generate a high return but induce a high volatility.

\vskip 4pt
\noindent
{\bf Acknowledgements.}
{\small Y. Liu acknowledges financial support from the National Natural Science Foundation of China (Grant No. 12401624), The Chinese University of Hong Kong (Shenzhen) research startup fund (Grant No. UDF01003336) and Shenzhen Science and Technology Program (Grant No. RCBS20231211090814028) and is partly supported by the Guangdong Provincial Key Laboratory of Mathematical Foundations for Artificial Intelligence (Grant No. 2023B1212010001). Z. Shen acknowledges financial support from the undergraduate research awards at The Chinese University of Hong Kong (Shenzhen).
The authors are grateful to Gongqiu Zhang and members of the research group on financial mathematics and risk management at The Chinese University of Hong Kong, Shenzhen for their useful feedback and conversations.}


\appendix

\section{\bf Proofs}\label{proofs}
In the proofs, we prove the more complicated case with $\alpha \neq 1$ in the utility family of Eq. \eqref{basic_U}. The log function case ($\alpha = 1$) holds as well.
\subsection{Proof of Proposition \ref{prop:generality}}
The proof of the first part of Proposition \ref{prop:generality} is by direct computation. For a PSAHARA utility $U$,
\begin{equation}
U\(x;\alpha_k,\beta_k,d_k,\gamma_k,u_k\) = \gamma_k \hat{U}\(x;\alpha_k,\beta_k,d_k\) + u_k, \quad x \in \(a_k,a_{k+1}\).
\end{equation}
For an increasing continuous piecewise linear function $h$, we can write $h(x) = Ax + B, \;x \in \mathcal{J}$, where $A > 0, B \in \R$, and $\mathcal{J} \subset \R$ is an interval segment such that $h$ is a linear function on $\mathcal{J}$. Hence, we have
\begin{equation}
\begin{aligned}
U\( Ax+B;\alpha_k, \beta_k,d_k,\gamma_k,u_k \) &= \gamma_k \hat{U}\(Ax+B;\alpha_k,\beta_k,d_k\) + u_k\\
&= A^{1-\alpha_k} \gamma_k \hat{U}\(x;\alpha_k,\frac{\beta_k}{A},\frac{d_k-B}{A}\) + u_k, \quad x \in \(a_k,a_{k+1}\) \cap \mathcal{J},
\end{aligned}
\end{equation}
which is also in the form of the PSAHARA utility. Further, $U\(Ax+B\)$ is continuous since $U\(x\)$ and $\Theta\(x\)$ are both continuous in $x$. Hence, $U\(Ax+B\)$ is a PSAHARA utility.
\subsection{Proof of Proposition \ref{prop:enve}}
The second part of Proposition \ref{prop:enve} is proved by a counter-example in Example \ref{eg:main example}. For the first part of Proposition \ref{prop:enve}, it suffices to prove that on the domain where the concave envelope differs from the original utility, $U^{**}$ is linear. We show this in Lemma \ref{lem:linear enve}, which comes from Lemma 1 in \cite{LLMV2024} and is highly related to Lemma 6.3 of \cite{BS2014}. The proof is referred to Lemma 5.1 of \cite{BC1994}.
\begin{lemma}\label{lem:linear enve}
Suppose that $U$ is continuous. The set $\mathcal{A} := \{ x \in \D : U \(x\) \neq U^{**}\(x\) \} $ is represented as a union of at most countably many disjoint open intervals, and $U^{**}$ is linear on each of the above intervals.
\end{lemma}
As each linear segment is of the SAHARA utility family, the concave envelope $U^{**}$ is a PSAHARA utility.
\subsection{Proof of Theorem \ref{thm:optimal wealth}}\label{sec:proof of optimal wealth}
Following Section \ref{sec:technical discussions}, this subsection is going to show the final step that the terminal wealth $I\(\Y_\xi \(x_0\) \xi_T\)$ is replicable, i.e., $I\(\Y_\xi \(x_0\) \xi_T\) \in \X^{x_0}$. To this end, define the process $\{V_t\}_{0\leq t\leq T }$:
\begin{equation}
V_t := \xi_t^{-1} \E\[ \xi_T I\(\Y_\xi(x_0) \xi_T\) | \F_t \].
\end{equation}
Regarding $V_t$ as a function of $\xi_t$ and applying Itô's formula to $\xi_t V_t$, we have
\begin{equation}\label{eq:dynamics of V}
\begin{aligned}
\d \(\xi_t V_t\) &= \(V_t + \xi_t \frac{\partial V_t}{\partial \xi_t} \)\d \xi_t + \frac12 \(2 \frac{\partial V_t}{\partial \xi_t} + \xi_t \frac{\partial^2 V_t}{\partial \xi_t^2} \)\d \langle \xi\rangle_t\\
&= \(V_t + \xi_t \frac{\partial V_t}{\partial \xi_t} \) \(- \xi_t \pmb{\theta}^\intercal\) \d \mathbf{W}_t + \Box \d t, \quad \forall t \in \[0,T\],
\end{aligned}
\end{equation}
where we do not show the drift term $\Box$ for brevity. 
On the other hand, for a wealth process $\{X_t\}_{0\leq t\leq T} $ controlled by $\{\pmb{\pi}_t\}_{0\leq t\leq T} $, we have
\begin{equation}\label{eq:dynamics of X}
\d \( \xi_t X_t \) = \[ \xi_t \pmb{\pi}_t^\intercal \pmb{\sigma} - \xi_t X_t \pmb{\theta}^\intercal \] \d \mathbf{W}_t.
\end{equation}
Comparing the diffusion coefficients of Eqs. \eqref{eq:dynamics of V} and \eqref{eq:dynamics of X} and letting
\begin{equation}\label{eq:form of optimal control}
    \pmb{\pi}_t = -\(\pmb{\sigma}_t \pmb{\sigma}_t^\intercal \)^{-1} \(\pmb{\mu}_t - r_t \mathbf{1}_m\) \xi_t \frac{\partial V_t}{\partial \xi_t},
\end{equation}
we have $X_t^{\pmb{\pi}} = V_t,\; 0\leq t\leq T$, and can hence replicate the optimal terminal wealth $I(\Y_\xi(x_0) \xi_T)$. Now, we proceed to give the explicit form of the optimal terminal wealth.\\
\noindent (1) Based on the martingale and duality method, the optimal terminal wealth is obtained from 
\begin{equation}
X_T^* = \arg \sup_{x \in \D} \[ U\(x\) - y^* \xi_T x \].
\end{equation}
According to Definition \ref{def:PSAHARA utility}, we solve the problem:\\
(i) For $k \in \{1, 2, \ldots, n\}$, if $y^* \xi_T \in (\gamma_k^+, \gamma_k^-)$, then
\begin{equation}
    X_T^* = \arg \sup_{x \in \D} \[ U \(x\) - y^* \xi_T x \] = a_k.
\end{equation}
(ii) For $k \in \{0, 1, \ldots, n\} $, if $y^*\xi_T \in (\gamma_{k+1}^-, \gamma_k^+) $, then
\begin{equation}
\begin{aligned}
X_T^* &= \arg \sup_{x \in \D} \[U\(x\) - y^* \xi_T x \]\\
&= I\(y^*\xi_T \)\\
&= d_k + \frac{1}{2}\( \(\frac{\gamma_k}{y^* \xi_T} \)^{\frac{1}{\alpha_k}} -\beta_k^2 \(\frac{\gamma_k}{y^* \xi_T} \)^{-\frac{1}{\alpha_k}}\).
\end{aligned}
\end{equation}
(iii) If $y^* \xi_T = \gamma_k^+$ = $\gamma_{k+1}^- $, (i.e. $U^{**}$ is linear on $[a_k,a_{k+1}]$), we have that $X_T^*$ can take any value in $[a_k,a_{k+1}]$. However, we know by previous discussions that 
\begin{equation}
    \p \( \left\{ y^* \xi_T = \gamma_k^+ = \gamma_{k+1}^-  \right\} \) = 0,\; \text{for any}\; \gamma_k^+, \gamma_{k+1}^- \in \R.
\end{equation}
Hence we can take $X_T^*$ by an arbitrary value. This finishes the proof of the first part of Theorem \ref{thm:optimal wealth}.

Now we show explicitly the wealth process that leads to terminal wealth $I\(y^* \xi_T\) $. Define $Z_{t,T} := \frac{\xi_T}{\xi_t}$, which is independent of $\xi_t$ and log-normally distributed.
Since we have already obtained the optimal terminal wealth, we proceed to replicate the strategy that achieves the terminal wealth. Observe that for any self-financing trading strategy $\{Y_t \}_{0\leq t \leq T}$ satisfying $\E \[ \int_0^T Y_s^2 \, \d s \] < \infty $, $\{\xi_t Y_t\}_{0\leq t\leq T}$ is an $\{\F_t\}_{0\leq t\leq T}$-martingale. 
Hence, we can compute the optimal wealth at time $t \in [ 0, T)$ by the martingale representation argument:
{\small
\begin{align}
X_t^* &= \xi_t^{-1} \E \[ \xi_T X_T^* \mid \F_t \] 
= \E \[Z_{t,T} X_T^* \mid \F_t \]\nonumber\\
&= \sum_{k=1}^n \E\[Z_{t,T}a_k \mathds{1}_{\left\{ y^* \xi_t Z_{t,T} \in \( \gamma_k^+, \gamma_k^-\) \right\}}\mid \F_t\]\nonumber\\
&\quad +\sum_{k=0}^n\E \[ Z_{t,T}\left( d_k + \frac{1}{2} \( \( \frac{\gamma_k}{y^* \xi_T} \)^{\frac{1}{\alpha_k}} - \beta_k^2 \( \frac{\gamma_k}{y^* \xi_T} \)^{-\frac{1}{\alpha_k}} \) \right) \mathds{1}_{\left\{y^*\xi_t Z_{t,T} \in \( \gamma_{k+1}^-, \gamma_k^+\) \right\} }\mid \F_t \]\nonumber\\
&= \sum_{k=1}^n \( e^{-\int_t^T r_s \, \d s}a_k \[\Phi\(g_0 \(\frac{\gamma_k^+}{y^* \xi_t} \)\)-\Phi\(g_0 \(\frac{\gamma_k^-}{y^* \xi_t} \)\) \]\)\nonumber\\
&\quad +\sum_{k=0}^n \left\{ e^{-\int_t^T r_s \, \d s}d_k \[\Phi\(g_0 \(\frac{\gamma_{k+1}^-}{y^* \xi_t} \)\)-\Phi\(g_0 \(\frac{\gamma_k^+}{y^* \xi_t} \)\) \]\right.\nonumber\\
&\quad \quad +e^{\(\frac{1}{\alpha_k}-1\)\int_t^T \(r_s + \frac{1}{2\alpha_k}\| \pmb{\theta}_s \|_2^2 \, \)\d s } \frac{1}{2} \( \frac{\gamma_k}{y^* \xi_t} \)^{\frac{1}{\alpha_k}} \( \Phi\( g_1 \( \frac{\gamma_{k+1}^-}{y^* \xi_t} \) \) - \Phi\( g_1 \( \frac{\gamma_k^+}{y^* \xi_t} \) \) \)\nonumber \\
&\quad \quad \left.+e^{\(\frac{1}{\alpha_k}+1\)\int_t^T \(-r_s + \frac{1}{2\alpha_k}\| \pmb{\theta}_s \|_2^2 \, \)\d s  } \( -\frac{1}{2}\)\beta_k^2 \( \frac{\gamma_k}{y^* \xi_t} \)^{-\frac{1}{\alpha_k}} \( \Phi\( g_2 \( \frac{\gamma_{k+1}^-}{y^* \xi_t} \) \) - \Phi\( g_2 \( \frac{\gamma_k^+}{y^* \xi_t} \) \) \) \right\}\mathds{1}_{\{\alpha_k \neq 0\}}\nonumber \\
&= X_t^{D} + X_t^{B} + X_t^{R} + X_t^{\bar{R}}. \label{eq:optimal wealth process}
\end{align}
}In the fourth equality in Eq. \eqref{eq:optimal wealth process}, we compute the expectation by the log-normal distribution of $Z_{t,T} $ and subtly combine and arrange the terms.
Hence, the optimal terminal wealth is as shown in Theorem \ref{thm:optimal wealth}.

\subsection{Proof of Theorem \ref{optimal portfolio}}\label{proof:optimal portfolio}
From Eq. \eqref{eq:form of optimal control}, we can write $X_t^*$ as a function of $\xi_t$: $X_t^* = X_t^*(\xi_t)$.
Apply Itô's formula to $X_t^* = X_t^*( \xi_t)$ and obtain:
\begin{equation}
\pmb{\pi}_t^* = -\(\pmb{\sigma}_t \pmb{\sigma}_t^\intercal \)^{-1} \(\pmb{\mu}_t - r_t \mathbf{1}_m\) \xi_t \frac{\partial X_t^* \(\xi_t \)}{\partial \xi_t}.
\end{equation}
For each $k \in \{ 0, 1, \ldots, n\}$, we compute as follows:\\
(1) If $\alpha_k = 0$, we have
{\small
\begin{align}
\hat{\pmb{\pi}}_{t,k}^{(1)}&=-\(\pmb{\sigma}_t \pmb{\sigma}_t^\intercal \)^{-1} \(\pmb{\mu}_t - r_t \mathbf{1}_m\) \xi_t \frac{\partial}{\partial \xi_t}\( e^{-\int_t^T r_s \, \d s}a_k \[\Phi\(g_0 \(\frac{\gamma_k^+}{y^* \xi_t} \)\)-\Phi\(g_0 \(\frac{\gamma_k^-}{y^* \xi_t} \)\) \]\)\nonumber\\
&= -e^{-\int_t^T r_s \, \d s}a_k \frac{\(\pmb{\sigma}_t \pmb{\sigma}_t^\intercal \)^{-1} \(\pmb{\mu}_t - r_t \mathbf{1}_m\) }{\sqrt{\int_t^T \| \pmb{\theta}_s \|_2^2\, \d s }} \[ \Phi' \( g_0 \( \frac{\gamma_k^+}{y^* \xi_t} \) \) - \Phi' \( g_0 \( \frac{\gamma_k^-}{y^* \xi_t} \) \) \].
\end{align}
}
(2) If $\alpha_k \neq 0$, we have
{\small
\begin{align}
\hat{\pmb{\pi}}_{t,k}^{(2)}&=-\( \pmb{\sigma}_t \pmb{\sigma}_t^\intercal \)^{-1} \( \pmb{\mu}_t - r_t \mathbf{1}_m \) \xi_t \frac{\partial}{\partial \xi_t} X_{t,k}^B\nonumber\\
&=-e^{-\int_t^T r_s \, \d s}d_k \frac{\(\pmb{\sigma}_t \pmb{\sigma}_t^\intercal \)^{-1} \(\pmb{\mu}_t - r_t \mathbf{1}_m\) }{\sqrt{\int_t^T \| \pmb{\theta}_s \|_2^2\, \d s }} \[ \Phi' \( g_0 \( \frac{\gamma_{k+1}^-}{y^* \xi_t} \) \) - \Phi' \( g_0 \( \frac{\gamma_k^+}{y^* \xi_t} \) \) \],\nonumber\\
\hat{\pmb{\pi}}_{t,k}^{(3)} &= -\(\pmb{\sigma}_t \pmb{\sigma}_t^\intercal \)^{-1} \(\pmb{\mu}_t - r_t \mathbf{1}_m\) \xi_t \frac{\partial}{\partial \xi_t} X_{t,k}^R\nonumber\\
&= \frac12\(\pmb{\sigma}_t \pmb{\sigma}_t^\intercal \)^{-1} \(\pmb{\mu}_t - r_t \mathbf{1}_m\) e^{\(\frac{1}{\alpha_k}-1\)\int_t^T \(r_s + \frac{1}{2\alpha_k}\| \pmb{\theta}_s \|_2^2 \, \)\d s }\( \frac{\gamma_k}{y^* \xi_t} \)^{\frac{1}{\alpha_k}}\left\{ \frac{1}{\alpha_k}\[ \Phi \( g_{1,k}\(\frac{\gamma_{k+1}^-}{y^* \xi_t} \) \) \right.\right. \nonumber\\
&\quad \left.\left. - \Phi \(g_{1,k} \( \frac{\gamma_k^+}{y^*\xi_t} \) \) \]  - \frac{1}{\sqrt{\int_t^T \| \pmb{\theta}_s \|_2^2 \, \d s}} \( \Phi' \(g_{1,k}\(\frac{\gamma_{k+1}^-}{y^* \xi_t} \) \) - \Phi' \( g_{1,k} \(\frac{\gamma_k^+}{y^* \xi_t} \) \) \) \right\},\nonumber\\
\hat{\pmb{\pi}}_{t,k}^{(4)} &=-\(\pmb{\sigma}_t \pmb{\sigma}_t^\intercal \)^{-1} \(\pmb{\mu}_t - r_t \mathbf{1}_m\) \xi_t \frac{\partial}{\partial \xi_t} X_{t,k}^{\bar{R}}\nonumber\\
&=\frac12 \beta_k^2 \(\pmb{\sigma}_t \pmb{\sigma}_t^\intercal \)^{-1} \(\pmb{\mu}_t - r_t \mathbf{1}_m\) e^{\(\frac{1}{\alpha_k}+1\)\int_t^T \(-r_s + \frac{1}{2\alpha_k}\| \pmb{\theta}_s \|_2^2 \, \)\d s  }\( \frac{\gamma_k}{y^* \xi_t} \)^{-\frac{1}{\alpha_k}}\left\{ \frac{1}{\alpha_k} \[ \Phi \(g_{2,k} \(\frac{\gamma_{k+1}^-}{y^* \xi_t} \) \) \right. \right. \nonumber \\
&\quad \left.\left. - \Phi \( g_{2,k} \(\frac{\gamma_k^+}{y^* \xi_t}\) \) \] +\frac{1}{\sqrt{\int_t^T \| \pmb{\theta}_s\|_2^2 \, \d s }} \( \Phi' \(g_{2,k}\(\frac{\gamma_{k+1}^-}{y^* \xi_t} \) \) -\Phi' \(g_{2,k}\(\frac{\gamma_k^+}{y^* \xi_t} \) \) \) \right\}.
\end{align}
}Observe that
{\small
\begin{align}
\hat{\pmb{\pi}}_{t,k}^{(3)} &= \(\pmb{\sigma}_t \pmb{\sigma}_t^\intercal \)^{-1} \(\pmb{\mu}_t - r_t \mathbf{1}_m\) X_{t,k}^R \( \frac{1}{\alpha_k} - \frac{ \Phi' \(g_{1,k}\(\frac{\gamma_{k+1}^-}{y^* \xi_t} \) \) - \Phi' \( g_{1,k} \(\frac{\gamma_k^+}{y^* \xi_t} \) \) }{\sqrt{\int_t^T \| \pmb{\theta}_s \|_2^2 \, \d s}\( \Phi \(g_{1,k}\(\frac{\gamma_{k+1}^-}{y^* \xi_t} \) \) - \Phi \( g_{1,k} \(\frac{\gamma_k^*}{y^* \xi_t} \) \) \)}  \),\nonumber\\
\hat{\pmb{\pi}}_{t,k}^{(4)} &= \(\pmb{\sigma}_t \pmb{\sigma}_t^\intercal \)^{-1} \(\pmb{\mu}_t - r_t \mathbf{1}_m\) \(-X_{t,k}^{\bar{R}}\) \( \frac{1}{\alpha_k} - \frac{ \Phi' \(g_{1,k}\(\frac{\gamma_{k+1}^-}{y^* \xi_t} \) \) - \Phi' \( g_{1,k} \(\frac{\gamma_k^+}{y^* \xi_t} \) \) }{\sqrt{\int_t^T \| \pmb{\theta}_s \|_2^2 \, \d s}\( \Phi \(g_{1,k}\(\frac{\gamma_{k+1}^-}{y^* \xi_t} \) \) - \Phi \( g_{1,k} \(\frac{\gamma_k^+}{y^* \xi_t} \) \) \)}  \).
\end{align}
}Hence, we can rearrange the terms to get
{\small
\begin{align}
\hat{\pmb{\pi}}_{t,k}^{(3)} + \hat{\pmb{\pi}}_{t,k}^{(4)} &= \frac{\(\pmb{\sigma}_t \pmb{\sigma}_t^\intercal \)^{-1} \(\pmb{\mu}_t - r_t \mathbf{1}_m\)}{\alpha_k} \sqrt{ \(X_{t,k}^R + X_{t,k}^{\bar{R}}\)^2 + b_{t,k}}+ \frac{\(\pmb{\sigma}_t \pmb{\sigma}_t^\intercal \)^{-1} \(\pmb{\mu}_t - r_t \mathbf{1}_m\)}{\sqrt{\int_t^T \| \pmb{\theta}_s \|_2^2 \, \d s}}\nonumber\\
&\quad\times \( X_{t,k}^R \frac{\Phi' \(g_{1,k}\(\frac{\gamma_{k+1}^-}{y^* \xi_t} \) \) - \Phi' \( g_{1,k} \(\frac{\gamma_k^+}{y^* \xi_t} \) \) }{\Phi \(g_{1,k}\(\frac{\gamma_{k+1}^-}{y^* \xi_t} \) \) - \Phi \( g_{1,k} \(\frac{\gamma_k^+}{y^* \xi_t} \) \)} + X_{t,k}^{\bar{R}} \frac{\Phi' \(g_{2,k}\(\frac{\gamma_{k+1}^-}{y^* \xi_t} \) \) - \Phi' \( g_{2,k} \(\frac{\gamma_k^+}{y^* \xi_t} \) \)}{\Phi \(g_{2,k}\(\frac{\gamma_{k+1}^-}{y^* \xi_t} \) \) - \Phi \( g_{2,k} \(\frac{\gamma_k^+}{y^* \xi_t} \) \)} \).
\end{align}
}Then we have the expressions in Theorem \ref{optimal portfolio}.

\subsection{Proof of Theorem \ref{thm:asymptotic analysis}}\label{proof of aymptotic analysis}
First, note that for $\gamma_0^+ = \infty$, we have for any $\xi_t$, 
\begin{equation}
\frac{\gamma_0^+}{y^*\xi_t} = \infty,\quad g_i \(\frac{\gamma_0^+}{y^*\xi_t}\) = -\infty, \quad
\Phi'\( g_i\(\frac{\gamma_0^+}{y^*\xi_t}\) \) = 0,\quad \Phi\( g_i\(\frac{\gamma_0^+}{y^*\xi_t}\) \) = 0,\quad i = 0, 1, 2.
\end{equation}
Similarly, for $\gamma_{n+1}^- = 0$, we have for any $\xi_t$,
\begin{equation}
\frac{\gamma_{n+1}^-}{y^*\xi_t} = 0,\quad g_i \(\frac{\gamma_{n+1}^-}{y^*\xi_t}\) = \infty, \quad
\Phi'\( g_i\(\frac{\gamma_{n+1}^-}{y^*\xi_t}\) \) = 0,\quad \Phi\( g_i\(\frac{\gamma_{n+1}^-}{y^*\xi_t}\) \) = 1,\quad i = 0, 1, 2.
\end{equation}
(a) For fixed $\gamma \in (0,\infty)$ and $y^* >0$, as $\xi_t \rightarrow 0$, we have
\begin{equation}
\frac{\gamma}{y^* \xi_t } \rightarrow \infty,\quad g_{i}\(\frac{\gamma}{y^* \xi_t} \) \rightarrow -\infty, \quad
\Phi' \(g_i\(\frac{\gamma}{y^* \xi_t} \) \) \rightarrow 0,\quad \Phi \(g_i\(\frac{\gamma}{y^* \xi_t} \) \) \rightarrow 0,\quad i=0,1,2.
\end{equation}
Recall the expression of $X_t^*$ in \eqref{part_t_wealth}. For $k \in \{1, \ldots, n\}$, as $\xi_t \rightarrow 0$,
\begin{equation}
X_{t,k}^D = e^{-\int_t^T r_s \, \d s} a_k \[ \Phi \( g_0\( \frac{\gamma_k^+}{y^* \xi_t} \) \) - \Phi \( g_0\( \frac{\gamma_k^-}{y^* \xi_t} \) \) \] = 0.
\end{equation}
Moreover, for $k\in \{0, 1, \ldots, n-1\}$, 
\begin{equation}
\Phi \( g_{i,k} \(\frac{\gamma_{k+1}^-}{y^*\xi_t}\) \) - \Phi\( g_{i,k}\( \frac{\gamma_k^+}{y^* \xi_t} \) \) = 0, \quad i = 0, 1, 2.
\end{equation}
Hence, we have
\begin{equation}
\begin{aligned}
X_t^D &\rightarrow 0,\quad
X_t^B \rightarrow d_n e^{-\int_t^T r_s \, \d s},\\
X_t^R 
&\rightarrow \infty,\quad
X_t^{\bar{R}} 
\rightarrow 0.
\end{aligned}
\end{equation}
Combining the parts above, we get
$
X_t^* \rightarrow \infty.
$
Then it follows from Theorem \ref{optimal portfolio} that
\begin{equation}
\pmb{\pi}_t^{(1)} \rightarrow \pmb{\infty}.
\end{equation}
Note that
\begin{equation}
\Phi' \( g_i\(\frac{\gamma}{y^*\xi_t} \) \) \rightarrow 0, \quad i = 0, 1, 2,
\end{equation}
for any $\gamma \in (0,\infty]$. Hence, we have
\begin{equation}
\pmb{\pi}_t^{(3)} \rightarrow \mathbf{0},\quad \pmb{\pi}_t^{(4)} \rightarrow \mathbf{0}.
\end{equation}
Now we deal with $\pmb{\pi}_t^{(2)}$:
{\small
\begin{align}
\pmb{\pi}_{t}^{(2)} &= -\frac{\(\pmb{\sigma}_t\pmb{\sigma}_t^\intercal\)^{-1} \(\pmb{\mu}_t - r_t \mathbf{1}_m \) }{\sqrt{\int_t^T \| \pmb{\theta}_s \|_2^2 \, \d s}} \sum_{k=0}^n \left\{ X_{t,k}^R \frac{ \Phi' \(g_{1,k} \(\frac{\gamma_{k+1}^-}{y^* \xi_t} \) \) - \Phi'\(g_{1,k} \(\frac{\gamma_k^+}{y^* \xi_t} \)\)}{ \Phi \(g_{1,k} \(\frac{\gamma_{k+1}^-}{y^* \xi_t} \) \) - \Phi\(g_{1,k} \(\frac{\gamma_k^+}{y^* \xi_t} \)\)} \right.\nonumber \\
&\quad \quad \quad \quad \quad \quad \quad \quad \quad +\left. X_{t,k}^{\bar{R}} \frac{ \Phi' \(g_{2,k} \(\frac{\gamma_{k+1}^-}{y^* \xi_t} \) \) - \Phi'\(g_{2,k} \(\frac{\gamma_k^+}{y^* \xi_t} \)\)}{ \Phi \(g_{2,k} \(\frac{\gamma_{k+1}^-}{y^* \xi_t} \) \) - \Phi\(g_{2,k} \(\frac{\gamma_k^+}{y^* \xi_t} \)\)} \right\}\times \mathds{1}_{\left\{ \alpha_k \neq 0 \right\}}\nonumber\\
&\rightarrow -\frac{\(\pmb{\sigma}_t\pmb{\sigma}_t^\intercal\)^{-1} \(\pmb{\mu}_t - r_t \mathbf{1}_m \) }{2\sqrt{\int_t^T \| \pmb{\theta}_s \|_2^2 \, \d s}}e^{\(-1+\frac{1}{\alpha_n}\)\int_t^T \(r_s + \frac{1}{2\alpha_n}\| \pmb{\theta}_s \|_2^2 \, \)\d s  } \( \frac{\gamma_n}{y^* \xi_t} \)^{\frac{1}{\alpha_n}} \Phi'\( g_{1,n} \( \frac{\gamma_n^+}{y^* \xi_t} \) \)\nonumber\\
&\rightarrow \mathbf{0}.
\end{align}
}Combining all above, we have
{\small
\begin{align}
\frac{\pmb{\pi}_t^*}{X_t^*}
&\rightarrow \(\pmb{\sigma}_t\pmb{\sigma}_t^\intercal\)^{-1} \(\pmb{\mu}_t - r_t \mathbf{1}_m \) \frac{1}{\alpha_n} \sqrt{ \(1 + \frac{X_{t,n}^{\hat{R}}}{X_{t,n}^R} \)^2 + \frac{b_{t,n}}{\(X_{t,n}^R\)^2} }\nonumber\\
&\rightarrow \frac{\(\pmb{\sigma}_t\pmb{\sigma}_t^\intercal\)^{-1} \(\pmb{\mu}_t - r_t \mathbf{1}_m \)}{\alpha_n} > \mathbf{0}.\quad \text{(component-wise)}
\end{align}
}Hence, $
\pmb{\pi}_t^* \rightarrow \pmb{\infty}.
$

\noindent
(b) For fixed $\gamma \in (0,\infty)$ and $y^* > 0$, as $\xi_t \rightarrow \infty$, we have
\begin{equation}
\frac{\gamma}{y^*\xi_t} \rightarrow 0, \quad g_i\( \frac{\gamma}{y^*\xi_t} \) \rightarrow \infty, \quad 
\Phi'\( g_i\(\frac{\gamma}{y^*\xi_t} \) \) \rightarrow 0, \quad \Phi\( g_i\( \frac{\gamma}{y^* \xi_t} \) \) \rightarrow 1, \quad i = 1,2,3.
\end{equation}
Hence, for $k \in \{ 0,1, \ldots, n\}$, 
\begin{equation}
\Phi' \( g_{i,k}\( \frac{\gamma_{k+1}^-}{y^*\xi_t} \) \) - \Phi' \( g_{i,k}\( \frac{\gamma_{k}^+}{y^*\xi_t} \) \) = 0-0 = 0, \quad i = 1,2,3.
\end{equation}
Also, for $k \in \{ 1, 2, \ldots, n\}$, 
\begin{equation}
\Phi \( g_{i,k}\( \frac{\gamma_{k+1}^-}{y^*\xi_t} \) \) - \Phi \( g_{i,k}\( \frac{\gamma_{k}^+}{y^*\xi_t} \) \) = 1 - 1 = 0,\quad i = 1,2,3.
\end{equation}
The only term left is
\begin{equation}
\Phi \( g_{i,0}\( \frac{\gamma_{1}^-}{y^*\xi_t} \) \) - \Phi \( g_{i,0}\( \frac{\gamma_{0}^+}{y^*\xi_t} \) \) = 1, \quad i = 1,2,3.
\end{equation}
It follows from Theorem \ref{thm:optimal wealth} that as $\xi_t \rightarrow \infty$,
\begin{equation}
\begin{aligned}
X_t^D &\rightarrow 0,\quad
X_t^B \rightarrow d_0 e^{-\int_t^T r_s \, \d s},\\
X_t^R 
&\rightarrow 0,\quad
X_t^{\bar{R}} 
\rightarrow -\infty.
\end{aligned}
\end{equation}
Similarly as in part (a), we have
\begin{equation}
\pmb{\pi}_t^{(1)} \rightarrow \pmb{\infty},\quad
\pmb{\pi}_t^{(3)} \rightarrow \mathbf{0},\quad
\pmb{\pi}_t^{(4)} \rightarrow \mathbf{0},
\end{equation}
and
{\small
\begin{align}
\pmb{\pi}_{t}^{(2)} &= -\frac{\(\pmb{\sigma}_t\pmb{\sigma}_t^\intercal\)^{-1} \(\pmb{\mu}_t - r_t \mathbf{1}_m \) }{\sqrt{\int_t^T \| \pmb{\theta}_s \|_2^2 \, \d s}} \sum_{k=0}^n \left\{ X_{t,k}^R \frac{ \Phi' \(g_{1,k} \(\frac{\gamma_{k+1}^-}{y^* \xi_t} \) \) - \Phi'\(g_{1,k} \(\frac{\gamma_k^+}{y^* \xi_t} \)\)}{ \Phi \(g_{1,k} \(\frac{\gamma_{k+1}^-}{y^* \xi_t} \) \) - \Phi\(g_{1,k} \(\frac{\gamma_k^+}{y^* \xi_t} \)\)} \right.\nonumber \\
&\quad \quad \quad \quad \quad \quad \quad \quad \quad +\left. X_{t,k}^{\bar{R}} \frac{ \Phi' \(g_{2,k} \(\frac{\gamma_{k+1}^-}{y^* \xi_t} \) \) - \Phi'\(g_{2,k} \(\frac{\gamma_k^+}{y^* \xi_t} \)\)}{ \Phi \(g_{2,k} \(\frac{\gamma_{k+1}^-}{y^* \xi_t} \) \) - \Phi\(g_{2,k} \(\frac{\gamma_k^+}{y^* \xi_t} \)\)} \right\}\times \mathds{1}_{\left\{ \alpha_k \neq 0 \right\}}\nonumber\\
&\rightarrow -\frac{\(\pmb{\sigma}_t\pmb{\sigma}_t^\intercal\)^{-1} \(\pmb{\mu}_t - r_t \mathbf{1}_m \) }{2\sqrt{\int_t^T \| \pmb{\theta}_s \|_2^2 \, \d s}}e^{\(-1-\frac{1}{\alpha_0}\)\int_t^T \(r_s - \frac{1}{2\alpha_0}\| \pmb{\theta}_s \|_2^2 \, \)\d s  } \beta_0^2 \( \frac{\gamma_0}{y^* \xi_t} \)^{-\frac{1}{\alpha_0}} \Phi'\( g_{2,0} \( \frac{\gamma_1^-}{y^* \xi_t} \) \)\nonumber\\
&\rightarrow \mathbf{0}.
\end{align}
}Combining all above, we have
{\small
\begin{align}
\frac{\pmb{\pi}_t^*}{X_t^*}
&\rightarrow -\( \pmb{\sigma}_t \pmb{\sigma}_t^\intercal \)^{-1} \( \pmb{\mu}_t - r_t \mathbf{1}_m \) \frac{1}{\alpha_0} \sqrt{ \( \frac{X_{t,0}^R}{X_{t,0}^{\bar{R}}} + 1 \)^2 + \frac{b_{t,0}}{\( X_{t,0}^{\bar{R}} \)^2} }\nonumber\\
&\rightarrow -\frac{\(\pmb{\sigma}_t \pmb{\sigma}_t^\intercal \)^{-1} \( \pmb{\mu}_t - r_t \mathbf{1}_m \)}{\alpha_0} < \mathbf{0}.\quad (\text{component-wise})
\end{align}
}Hence, $
\pmb{\pi}_t^* \rightarrow \pmb{\infty}.
$
\subsection{Proof of Theorem \ref{thm:incentive portfolio}}
The concave envelope of \eqref{eq:incentive utility} is given by
{\small
\begin{equation}\label{eq:incentive envelope}
    \Tilde{U}\(x\) := \left\{
    \begin{aligned}
    & v^{1-\alpha} \hat{U}\(x; \alpha, \frac{\beta}{v^2}, d\) + u_0, && x < a_1;\\
    & v^{1-\alpha} \hat{U}'\(a_1; \alpha, \frac{\beta}{v^2},d \) \( x - a_1\) + \hat{U}\(a_1; \alpha, \frac{\beta}{v^2}, 0\), && a_1 \leq x < a_2;\\
    & \(w+v\)^{1-\alpha} \hat{U}\(x; \alpha, \frac{\beta}{\(w+v\)^2}, \frac{wB_T}{w+v}\) + u_2, && x \geq a_2.\\
    \end{aligned}
    \right.
\end{equation}
}From the arguments in Section \ref{sec:technical discussions}, Problem \eqref{eq:incentive problem} is equivalent to the problem replaced with its concave envelope:
\begin{equation}
    \max_{\pmb{\pi}\in \mathcal{V}} E\[ \Tilde{U}\(X_T\) \].
\end{equation}
Applying Theorem \ref{optimal portfolio} to $\Tilde{U}$, we get Theorem \ref{thm:incentive portfolio}.

\section{Volatility estimation}\label{appendix:empirical}
There are a lot of models to study the properties of the volatility, including the stochastic and implied volatility models; see e.g., \cite{BBF2002}, \cite{GHLOW2010}, \cite{DTY2016} and \cite{CKN2018}. Since this study does not focus on the market modeling, we do not adopt the complicated models. We describe three estimation methods of $\{\pmb{\sigma}_t \}_{0 \leq t \leq T}$ in the illustration on the performance of our optimal portfolio.

\subsection{Historical Volatility}
We assume that $\{\pmb{\sigma}_t \}_{0\leq t\leq T}$ in \eqref{wealthprocess} directly depicts the historical returns of the risky assets. From the in-data sample indicated in Section \ref{sec:empirical}, we get the covariance matrix of the returns of the stocks $\pmb{\Sigma}$. Since $\pmb{\Sigma}$ is symmetric, we can always find the matrix $\pmb{\sigma}$ satisfying 
$\pmb{\sigma}\pmb{\sigma}^\intercal = \pmb{\Sigma}.
$ 
We use the $i$th row of matrix $\pmb{\sigma}$ as the volatility of $i$th stock. That is,
$
\d S_{i,t} = \mu_{i} S_{i,t} \d t + S_{i,t} \pmb{\sigma}_i \d \mathbf{W}_t,
$ where $\mu_i$ is the expected return of $i$th stock and $\pmb{\sigma}_i$ is the $i$th row of $\pmb{\sigma}$.

\subsection{Implied Volatility}\label{sec:implied volatility}
Although the method of historical volatility is easy for implementation, it fails to capture some characteristics of the Black--Scholes model.
Another way to calculate $\{\pmb{\sigma}_t \}_{0\leq t \leq T}$ is through the Black--Scholes formula for European option prices. In a multi-asset Black--Scholes model, let $P_{i,t}^K$ denote the price of a put option signed on the risky asset $S_{i,t}$ with a strike price $K$. According to \cite{BS1972}, we have (with a slight abuse of notation)
{\small
\begin{equation}\label{eq:black scholes put}
P_{i,t}^K = Ke^{-\int_t^T r_s \, \d s}\Phi\(-d_2 \) - S_{i,t} \Phi\(-d_1\),
\end{equation}
}where
{\small
\begin{equation}
d_1 := \frac{1}{\sqrt{\int_t^T \| \pmb{\sigma}_{i,s} \|_2^2 \, \d s}}  \( \log\(\frac{S_t}{K}\) + \int_t^T \( r_s + \frac{\| \pmb{\sigma}_{i,s}\|_2^2 }{2} \) \, \d s \),\quad
d_2 := d_1 - \sqrt{ \int_t^T \| \pmb{\sigma}_{i,s}\|_2^2 \, \d s }.
\end{equation}
}Hence, we can calculate the norm of the $i$th row of $\pmb{\sigma}_t$ and reallocate the weight to each component according to the correlation among the risky assets. However, in security markets, the value of $\| \pmb{\sigma}_{i,t} \|_2^2$ varies corresponding to different $K$. Such phenomena are referred to as the \textit{volatility smile} or \textit{volatility smirk} based on different scenarios.
An interpretation of the volatility ``smile"  is the imperfection of financial markets. For example, put options on indexes usually incur the volatility ``smirk". Empirically, put options with higher strike prices imply lower volatility. Hence, the smirk can be viewed as the consequence of the bid--ask spread; see \cite{PRS2002}. Since most traders long the indexes, they would like to hedge the risk by longing the put option on index-tracking ETFs. To hedge the extreme risks only, they tend to choose the put option with lower strike prices. This imbalance in demand results in an increase to put option prices with lower strike prices and a drop to the prices with higher strike prices. Hence, it is reasonable to estimate the real volatility by a weighted average of the implied volatilities calculated from put options with different levels of strike prices. As shown in \cite{EG2002}, the arithmetic average of implied volatility is good enough for prediction purposes. Since this study does not focus on econometrics, we simply choose the $L^2$ norm of the $i$th asset's implied volatility to be $\| \pmb{\sigma}_{i,t} \|_2 $. We use the historical return to get the correlation matrix of the risky assets, denoted by $\mathbf{C}$. Then we find the matrix $\mathbf{B}$ satisfying
\begin{equation}\label{eq:corr matrix}
\mathbf{B} \mathbf{B}^\intercal = \mathbf{C}.
\end{equation}
After scaling each row of $\mathbf{B}$ to coincide with the norm of implied volatilities, we have the desired volatility matrix.

\subsection{MLE Volatility}
The maximum likelihood estimator (MLE) method is another approach to estimate the market parameters. Similarly as in Appendix \ref{sec:implied volatility}, we calculate $\|\pmb{\sigma}_{i,t}\|_2^2$ for $i = 1, \ldots, m$ and scale the correlation matrix $\mathbf{B}$ given by \eqref{eq:corr matrix} accordingly. Let $p_{i,k}$ denote the return of the $i$th risky asset on day $k$. According to Section 9.3.2 in \cite{CLM1997}, the MLE estimator of $\|\pmb{\sigma}_{i,t}\|_2^2$ is given by
\begin{equation}
\|\hat{\pmb{\sigma}}_i\|_2^2 = \frac{1}{nh} \sum_{k=1}^n \(p_{i,k} - \frac{1}{nh}\sum_{k=1}^n p_{i,k} \),\quad i = 1,\ldots, m,
\end{equation}where $n$ denotes the number of total trading days and $h$ is the time gap between each trading day. In our context, $n = 8 \times 252$ and $h = 1 / 252$.
\end{document}